\documentclass[reqno, 12pt]{amsart}
\usepackage{a4wide, amsfonts, amssymb, amsmath, amsthm, mathrsfs, enumerate, enumitem, setspace, url}
\usepackage[utf8]{inputenc}
\usepackage{graphicx}
\usepackage{tikz-cd} 
	\usetikzlibrary{cd}
\usepackage[all]{xy}
\usepackage{mathtools}
\usepackage{lmodern}
\usepackage[french,english]{babel}

\usepackage{color}
\usepackage{enumitem}
\usepackage{algpseudocode}
\usepackage{hyperref} 
	\definecolor{darkred}{rgb}{0.5,0,0}
	\definecolor{darkgreen}{rgb}{0,0.5,0}
	\definecolor{darkblue}{rgb}{0,0,0.5}
	\hypersetup{linkcolor=darkblue,filecolor=darkgreen,urlcolor=darkred,citecolor=darkblue}

\usepackage{comment}
\usepackage{longtable}
\usepackage[T2A,T1]{fontenc}
\DeclareSymbolFont{cyrillic}{T2A}{cmr}{m}{n}
\DeclareMathSymbol{\Sha}{\mathalpha}{cyrillic}{216}

\theoremstyle{plain}
\newtheorem{theorem}{Theorem}[section]
\newtheorem*{theorem*}{Theorem}
\newtheorem{proposition}[theorem]{Proposition}
\newtheorem{lemma}[theorem]{Lemma}

\theoremstyle{remark}
\newtheorem{remark}[theorem]{Remark}

\newtheorem*{acknowledgements}{Acknowledgements}

\theoremstyle{definition}
\newtheorem{definition}[theorem]{Definition}

\numberwithin{equation}{section}

\newcommand{\NN}{\mathbb{N}}
\newcommand{\ZZ}{\mathbb{Z}}
\newcommand{\QQ}{\mathbb{Q}}

\newcommand{\RR}{\mathbb{R}}
\newcommand{\CC}{\mathbb{C}}

\newcommand{\Fp}{\mathbb{F}_p}

\renewcommand{\AA}{\mathbb{A}}

\DeclareMathOperator{\Gal}{Gal}
\DeclareMathOperator{\Pic}{Pic}
\DeclareMathOperator{\Br}{Br}
\DeclareMathOperator{\inv}{inv}

\DeclareMathOperator{\N}{N}

\renewcommand{\epsilon}{\varepsilon}

\newcommand{\Q}{\mathbb{Q}}
\newcommand{\Z}{\mathbb{Z}}
\renewcommand{\mod}{\operatorname{mod}}
\renewcommand{\P}{\mathbb{P}}

\newcommand{\NAM}{N_{\textbf{A},\textbf{M}}}

\newcommand{\littletaller}{\mathchoice{\vphantom{\big|}}{}{}{}}

\newcommand\res[2]{{
  \left.\kern-\nulldelimiterspace 
  #1 
  \littletaller 
  \right|_{#2} 
  }}

\DeclarePairedDelimiter\abs{\lvert}{\rvert}
\DeclarePairedDelimiter\norm{\lVert}{\rVert}

\makeatletter
\let\oldabs\abs
\def\abs{\@ifstar{\oldabs}{\oldabs*}}
\let\oldnorm\norm
\def\norm{\@ifstar{\oldnorm}{\oldnorm*}}
\makeatother
\renewcommand{\O}{\mathcal{O}}

\newcommand{\inva}[1]{\inv_p(\mathcal{A}(#1))}
\newcommand\link[1]{\hyperref[#1]{#1}}
\begin{document}
\onehalfspacing
\title[Diagonal del Pezzo Surfaces of degree 2 with a Brauer-Manin obstruction]{Diagonal del Pezzo Surfaces of degree 2 with a Brauer-Manin obstruction}

\author{Harry C. Shaw}
	\address{Department of Mathematical Sciences, University of Bath, Claverton Down, Bath, BA2 7AY, UK}
	\email{hcs50@bath.ac.uk}

\date{\today}
\begin{abstract}
In this paper we give an asymptotic formula for the quantity of diagonal del Pezzo surfaces of degree $2$ which have a Brauer-Manin obstruction to the Hasse principle when ordered by height.
\end{abstract}
\maketitle

\setcounter{tocdepth}{1}
\tableofcontents

\section{Introduction}
Given a variety $X$ over some number field $k$, a fundamental question to ask is whether $X(k)\neq \varnothing$? If $X(k)\neq \varnothing$, then since $X(k)\subseteq X(\mathbb{A}_k)$ we must also have that $X(\mathbb{A}_k)\neq\varnothing$. One may ask whether the reverse implication holds, that is, does $X(\mathbb{A}_k)\neq\varnothing \implies X(k)\neq\varnothing$? If this implication holds, we say that $X$ satisfies the \textit{Hasse principle}. It is well-known that in general this fails to hold. For example, in \cite[Ex.~6]{On_the_arithmetic_of_del_Pezzo_surfaces_of_degree_2} it was shown that the surface
\begin{equation*}
    -126x_0^4-91x_1^4+78x_2^4=w^2,
\end{equation*}
has no $\Q$-point, but it is everywhere locally soluble. In this example, the failure of the Hasse principle is explained by the \textit{Brauer-Manin obstruction}. It has been conjectured by Colliot-Th\'el\`ene that for rationally connected $k$-varieties $X$, the Brauer-Manin obstruction is the only obstruction to $X$ satisfying the Hasse principle (see \cite[Conj.~14.1.2]{Brauer-Grot-Book}). As such it is of interest to determine how often a surface in this family has a Brauer-Manin obstruction to the Hasse principle.\par
In this paper we study the Brauer-Manin obstruction to the Hasse principle for diagonal del Pezzo surfaces of degree 2 over $\Q$ with integer coefficients (a subfamily of rationally connected $\Q$-varieties). Namely the surfaces
\begin{equation*}
    S_\textbf{a}: a_0x_0^4+a_1x_1^4+a_2x_2^4=w^2 \subseteq \P_\QQ(1,1,1,2),
\end{equation*}
where $a_i\in \Z\setminus\{0\}$. We are able to find an asymptotic formula for the number of such surfaces which have a Brauer-Manin obstruction to the Hasse principle:
\begin{theorem}
\label{thm: main theorem - asymptotic formula}
 There exists a constant $A>0$ such that
 \begin{equation*}
     \#\left\{ 
     \textbf{a}\in (\Z\setminus\{0\})^3 :\begin{array}{l}\abs{a_i}\leq T \text{ for all } i\in\{0,1,2\},\\
     S_\textbf{a} \text{ has a Brauer-Manin obstruction}\\
     \text{to the Hasse principle}
     \end{array} \right\} \sim A(T\log T)^\frac{3}{2}.
 \end{equation*}
\end{theorem}
In fact, if we define the natural partition of the coefficients as follows:
\begin{equation*}
    N^{\Br}_{=\Box}(T):=\left\{\textbf{a}\in N^{\Br}(T): -a_0a_1a_2\in \Q^{\times 2}\right\},
\end{equation*}
\begin{equation*}
    N^{\Br}_{=-\Box}(T):=\left\{\textbf{a}\in N^{\Br}(T): -a_0a_1a_2\in -\Q^{\times 2}\right\},
\end{equation*}
\begin{equation*}
    N^{\Br}_{\neq\pm\Box}(T):=\left\{\textbf{a}\in N^{\Br}(T): -a_0a_1a_2\not\in \pm\Q^{\times 2}\right\},
\end{equation*}
then Theorem~\ref{thm: main theorem - asymptotic formula} arises as a special case of the following result:
\begin{theorem}
\label{thm: main theorem split into squares and non squares}
There exists $A,B>0$ such that
\begin{equation*}
       \#N^{\Br}_{=\Box}(T) = O(T^\frac{3}{2}\log T),
\end{equation*}
\begin{equation*}
       \#N^{\Br}_{=-\Box}(T) \sim A(T\log T)^\frac{3}{2},
\end{equation*}
\begin{equation*}
       \#N^{\Br}_{\neq\pm\Box}(T) \sim BT^\frac{3}{2}(\log T)^\frac{9}{8}.
\end{equation*}
\end{theorem}
Recently there has been similar work carried out on a range of surfaces, for example; see \cite{GvirtzDamián2022Qaod} and \cite{TimSantens} for K3-surfaces, \cite{BretècheR.dela2014DoCs} and \cite{RomeNick2019APPo} for Ch\^atelet surfaces, and \cite{JahnelJörg2016Otno} and \cite{MitankinVladimir2022Rpod} for del Pezzo surfaces. As far as the author is aware this is the first such result for any family of del Pezzo surfaces of degree $2$. Furthermore, while the majority of papers were only able to obtain the correct order of magnitude, we are able to obtain an asymptotic formula.\par
In \cite{On_the_arithmetic_of_del_Pezzo_surfaces_of_degree_2} the Brauer groups of the surfaces $S_\textbf{a}$ were calculated, in which it was proved that 100\% of these surfaces satisfy $\Br S_\textbf{a}/\Br\Q\cong \Z/2\Z$, generated by an explicit quaternion algebra $\mathcal{A}$. However, there is not a uniform generator across the family of these surfaces, adding further difficulty to our problem. Precisely, we prove the following result:
\begin{theorem}
\label{thm: BrS/Brk=0}
Let $a_0,a_1,a_2$ be algebraically independent transcendental elements over $\Q$ and let $k:=\Q(a_0,a_1,a_2)$. Let $\mathcal{S}\subseteq \P_k(1,1,1,2)$ be the surface given by
\begin{equation*}
    a_0x_0^4+a_1x_1^4+a_2x_2^4=w^2.
\end{equation*}
Then we have $\Br \mathcal{S}/\Br k=0$, however $H^1(k, \Pic \mathcal{S}_{\overline{k}})=\Z/2\Z$.
\end{theorem}
This prevents us from being able to more easily construct families of the surfaces $S_\textbf{a}$ which have a Brauer-Manin obstruction. We are however able to adapt the methods used in \cite{TimSantens} to mitigate some of this difficulty.

\subsection{Outline of the paper}
In \S\ref{sec: Calculating the Local Invariant Maps} we begin by constructing a representative $\mathcal{A}$ of an element of the Brauer group using the methods given in \cite{TimSantens}. We then proceed to calculate certain local invariant maps for $\mathcal{A}$ using similar methods as given in \cite{BrightM.J.2011TBoo} and \cite{TimSantens}. We end this section by determining how said maps change when the coefficients of the surface vary.\par
In \S\ref{sec: counting} we will prove the main theorem by following the methods used in \cite[\S5]{TimSantens}. We start by proving that the majority of surfaces with a Brauer-Manin obstruction actually have a Brauer-Manin obstruction induced by $\mathcal{A}$.
We then prove that $50\%$ of the surfaces for which $\inv_p(\mathcal{A}(-))$ is constant for all places $p$ and is everywhere locally soluble have a Brauer-Manin obstruction induced by $\mathcal{A}$. We lastly convert this set into a certain sum, and after simplifying the sum, provide an asymptotic formula which will complete the proof of the main theorem.\par
In \S\ref{sec: Uniform Formula} we prove that there is no uniform formula for $\mathcal{A}$ using the methods given in \cite{UematsuTetsuya} and \cite{TimSantens}. While this section does not directly effect the rest of the paper it provides an explanation as to the difficulty of some of the proofs.
\subsection{Notation and conventions used}
\label{subsec: Notation and conventions used}
Where possible we will use the notation given in \cite{TimSantens} to make the similarities between the papers clearer.
\begin{definition}
For $\textbf{a}=(a_0,a_1,a_2)\in (\Q^{\times})^3$ we let $S_{\textbf{a}}\subseteq\P_\QQ(1,1,1,2)$ denote the surface    
given by the equation
\begin{equation*}
    a_0x_0^4+a_1x_1^4+a_2x_2^4=w^2,
\end{equation*}
and let $X_\textbf{a}\subseteq \P_\QQ^3$ denote the $K3$ surface given by the equation
\begin{equation*}
    a_0x_0^4+a_1x_1^4+a_2x_2^4=x_3^4.
\end{equation*}
We let $\pi_\textbf{a}:X_\textbf{a}\rightarrow S_\textbf{a}; [x_0:x_1:x_2:x_3]\mapsto[x_0:x_1:x_2:x_3^2]$. We will write $\theta_\textbf{a}:=-a_0a_1a_2$.
\end{definition}
\begin{acknowledgements}
The author would like to thank Daniel Loughran for suggesting the problem and his tireless support throughout the process. The author would also like to thank Tim Santens for his useful comments on an earlier revision of this paper.
\end{acknowledgements}

\section{Calculating the local invariant maps}
\label{sec: Calculating the Local Invariant Maps}
\subsection{Brauer-Manin obstruction}
\label{subsec: Brauer-Manin obstruction}
In this section we will briefly recall the Brauer-Manin obstruction. Firstly recall from \cite[Def.~13.1.7]{Brauer-Grot-Book} that for any place $p$ of $\Q$ we have a homomorphism
\begin{equation*}
    \inv_p:\Br \Q_p \rightarrow \Q/\Z.
\end{equation*}
For a quaternion algebra $(a,b)\in \Br \Q_p[2]$ this map is given by
\begin{equation*}
    \inv_p((a,b))=\rho((a,b)_p),
\end{equation*}
where $(-,-)_p$ denotes the Hilbert symbol with respect to $p$, and $\rho:\mu_2\xrightarrow{\sim} (\Q/\Z)[2]$ is the unique isomorphism given by $-1\mapsto\frac{1}{2}$. \par 
For a smooth variety $S$ over $\Q$, we define the Brauer group of $S$ to be $\Br S:= \operatorname{H}^2_{\'et}(S,\mathbb{G}_{m,S})$. Let $B\subseteq \Br S$ be any non-empty subset, and let
\begin{equation*}
    S(\AA_\Q)^{B}:=\bigcap_{\mathcal{A}\in B}\left\{ (s_p)_p\in S(\AA_\Q): \sum_p\inv_p(\mathcal{A}(s_p))=0\right\}.
\end{equation*}
We say $S$ \textbf{has a Brauer-Manin obstruction to the Hasse principle} if $S(\AA_\Q)\neq \varnothing$ but $S(\AA_\Q)^{\Br S}=\varnothing$. For any $\mathcal{A}\in \Br S$, we say $S$ \textbf{has a Brauer-Manin obstruction induced by $\mathcal{A}$} if $S(\AA_\Q)\neq \varnothing$ but $S(\AA_\Q)^{\mathcal{A}}=\varnothing$. In particular, from \cite[Thm.~13.3.2]{Brauer-Grot-Book} we know that $S(\Q)\subseteq S(\AA_\Q)^{\Br S}$, and hence if $S$ has a Brauer-Manin obstruction to the Hasse principle, then it cannot satisfy the Hasse principle. Throughout the rest of the paper we will simply write \textit{Brauer-Manin obstruction} when we mean \textit{Brauer-Manin obstruction to the Hasse principle}.

\subsection{Constructing an element of the Brauer group}
\label{subsec: Constructing an element of the Brauer group}
In this section we will construct multiple elements of $\Br S_\textbf{a}$, which we will use in \S\ref{subsec: Computing the local invariant maps} and \S\ref{sec: counting}.
\begin{definition}
 Let  $Y_\textbf{a}\subseteq \P^3_\Q$ denote the smooth quadratic surface given by the equation
 \begin{equation*}
    a_0t_0^2+a_1t_1^2+a_2t_2^2=t_3^2.
\end{equation*}
Let $\phi_\textbf{a}:S_\textbf{a}\rightarrow Y_\textbf{a}$ denote the morphism given by $[x_0:x_1:x_2:w]\mapsto[x_0^2:x_1^2:x_2^2:w]$.
\end{definition}
In particular, if $S_\textbf{a}$ is everywhere locally soluble, then so is $Y_\textbf{a}$, and hence by the Hasse-Minkowski Theorem $Y_\textbf{a}(\Q)\neq\varnothing$.
\begin{definition}
Assume $\theta_\textbf{a}\in \Q^{\times 2}$. Let $a_3:=-1$, $x_3^2:=w$ and $\{i,j,k,\ell\}:=\{0,1,2,3\}$. Let $Z^{ijk}_\textbf{a}\subseteq \P^2_\Q$ denote the conic given by
 \begin{equation*}
     a_iZ_i^2+a_jZ_j^2+a_kZ_k^2=0.
 \end{equation*}
Fix a choice of $\sqrt{\theta_\textbf{a}}$. Let $\gamma_{\sqrt{\theta_\textbf{a}}}:S_\textbf{a} \rightarrow Z^{ijk}_\textbf{a}$ denote the rational map given by
\begin{align*}
    Z_i=x_i^2x_k^2-\frac{\sqrt{\theta_\textbf{a}}}{a_ia_k}x_j^2x_\ell^2, \text{  }Z_j=\frac{\sqrt{\theta_\textbf{a}}}{a_ja_k}x_i^2x_\ell^2+x_j^2x_k^2, \text{  } Z_k=x_k^4+\frac{a_\ell}{a_k}x_{\ell}^4.
\end{align*}   
\end{definition}
We now adapt the construction of elements of $\Br X_\textbf{a}$ in \cite[Prop.~3.1]{TimSantens} to our case:
\begin{proposition}
\label{prop: quaternion algebra definition from BrightM.J.2011TBoo}
 Let $S_\textbf{a}$ be everywhere locally soluble.
 \begin{enumerate}
     \item Let $P=[y_0:y_1:y_2:y_3]\in Y_\textbf{a}(\Q)$. Let $g_\textbf{a}$ be the polynomial defining the tangent plane to $Y_\textbf{a}$ at $P$, that is
 \begin{equation*}
     g_\textbf{a}:= a_0y_0t_0+a_1y_1t_1+a_2y_2t_2-y_3t_3.
 \end{equation*}
 Let $f_\textbf{a}$ be the pullback of $g_\textbf{a}$ via $\phi_\textbf{a}$ to $S_\textbf{a}$, that is
 \begin{equation*}
f_\textbf{a}:=\phi_\textbf{a}^*g_\textbf{a}=a_0y_0x_0^2+a_1y_1x_1^2+a_2y_2x_2^2-y_3w.
 \end{equation*}
 Then the quaternion algebra 
 \begin{equation*}
     \mathcal{A}_\textbf{a}:=(-a_0a_1a_2,f_\textbf{a}/x_0^2)
 \end{equation*}
 lies in $\Br S_\textbf{a}$. Furthermore, the class of $[\mathcal{A}_\textbf{a}]$ in $\Br S_\textbf{a}/\Br\Q$ is independent of the choice of $P$.
 \item Assume $\theta_\textbf{a}\in \Q^{\times 2}$. Let $a_3:=-1$, $x_3^2:=w$ and $\{i,j,k,\ell\}:=\{0,1,2,3\}$. Then the rational map $\gamma_{\sqrt{\theta_\textbf{a}}}:S_\textbf{a} \rightarrow Z^{ijk}_\textbf{a}$ can be extended to a morphism, and hence by the Hasse-Minkowski Theorem $Z^{ijk}_\textbf{a}(\Q)\neq\varnothing$. Let $P\in Z^{ijk}_\textbf{a}(\Q)$, and let $h_\textbf{a}$ be the tangent line to $Z^{ijk}_\textbf{a}$ at $P$. Then the quaternion algebra 
\begin{equation*}
    \mathcal{B}_{ijk,\sqrt{\theta_\textbf{a}}}:=\left(-a_ia_j\sqrt{\theta_\textbf{a}},\gamma_{\sqrt{\theta_\textbf{a}}}^*\frac{h_\textbf{a}}{Z_k}\right)
\end{equation*}
lies in $\Br S_\textbf{a}$. Furthermore, the class of $[\mathcal{B}_{ijk,\sqrt{\theta_\textbf{a}}}]$ in $\Br S_\textbf{a}/\Br\Q$ is independent of choice of $P$.
 \end{enumerate}

\end{proposition}
\begin{proof}
 The proof of this result follows from the proof of \cite[Prop.~3.1]{TimSantens} by appropriately changing the variables.   
\end{proof}

\begin{remark}
\label{rem: pullback of algbra is algebra from X}
Let $\mathcal{A}'_\textbf{a}\in \Br X_\textbf{a}$ denote the element constructed in \cite[Prop.~3.1.(i)]{TimSantens}, and let $\mathcal{B}_{ijk,\sqrt{\theta_\textbf{a}}}'\in \Br X_\textbf{a}$ denote the element constructed in \cite[Prop.~3.1.(ii)]{TimSantens}. Under the pullback of the map $\pi_\textbf{a}:X_\textbf{a}\rightarrow S_\textbf{a}$ we clearly have
\begin{equation*}
\pi_\textbf{a}^*\mathcal{A}_\textbf{a}=\mathcal{A}_\textbf{a}', \text{ and } \pi_\textbf{a}^*\mathcal{B}_{ijk,\sqrt{\theta_\textbf{a}}}=\mathcal{B}_{ijk,\sqrt{\theta_\textbf{a}}}'.
\end{equation*}
This fact will be used throughout the remaining sections to help simplify several proofs.   
\end{remark}

\subsection{Computing the local invariant maps}
\label{subsec: Computing the local invariant maps}
In this section we analyse the local invariant maps associated to the quaternion algebra $\mathcal{A}_\textbf{a}$. Throughout this section to simplify notation we will let $S:=S_\textbf{a}$, $\mathcal{A}:= \mathcal{A}_\textbf{a}$, $f:=f_\textbf{a}$ and $Y:=Y_\textbf{a}$.
\subsubsection{Existence of local solutions}
Before analysing the local invariant maps we provide criteria for the existence of local solution of $S$:
\begin{lemma}
\label{lem: local solubility of S}
 Let $p>33$ be a prime, then $S(\Q_p)\neq \varnothing$ if and only if one of the following holds:
 \begin{enumerate}
     \item \label{case: a_0 is a square in p-adics} $a_i\in \Q_p^{\times2}$ for some $i\in\{0,1,2\}$;
     \item \label{case: -a_0/a_1 is a 4th power in p-adics}$-\frac{a_i}{a_j}\in \Q_p^{\times4}$ for some $i\neq j\in\{0,1,2\}$;
     \item \label{case:a_0,a_1 have even class mod 4} there exists $i\neq j\in\{0,1,2\}$ and $s\in\{0,2\}$ such that $\nu_p(a_i)\equiv\nu_p(a_j)\equiv s\mod 4$;
     \item \label{case: a_0,a_1,a_3 have same class mod 4}  
     we have $\nu_p(a_0)\equiv\nu_p(a_1)\equiv\nu_p(a_2)\mod 4$.
 \end{enumerate}
\end{lemma}
\begin{proof}
The proceeding proof is adapted from the proof of \cite[Lem.~3.3]{TimSantens}: Assume $S(\Q_p)\neq \varnothing$. Let $Q=[p^{n_0}\mu_0:p^{n_1}\mu_1:p^{n_2}\mu_2:p^{n_3}\mu_3]\in S(\Z_p)$, where $n_i=0$ for some $i\in\{0,1,2,3\}$ and $\mu_j\in \Z_p^{\times}$ for all $j$. By definition we have
\begin{equation}
\label{eq: local solutions eq 1}
a_0p^{4n_0}\mu_0^4+a_1p^{4n_1}\mu_1^4+a_2p^{4n_2}\mu_2^4-p^{2n_3}\mu_3^2=0.
\end{equation}
Now if $m:=2n_3\leq 4n_i +\nu_p(a_i)$ for all $i\in\{0,1,2\}$, then multiplying \eqref{eq: local solutions eq 1} by $p^{-m}$ gives
\begin{equation*}
    a_0p^{4n_0-m}\mu_0^4+a_1p^{4n_1-m}\mu_1^4+a_2p^{4n_2-m}\mu_2^4-\mu_3^2=0.
\end{equation*}
Now consider this equation modulo $p$. Since $\mu_3\in \Z_p^{\times}$, there must exists at least one $j\in\{0,1,2\}$ such that $a_jp^{4n_j-m}\mu_j^4\not\equiv 0\mod p$. If exactly one such $j$ exists, then $a_jp^{4n_j-m}$ is a square modulo $p$. In particular since $\nu_p(a_j)\in 2\Z$, by Hensel's Lemma we have $a_j\in \Q_p^{\times2}$. If more than one such $j$ exists, let $J\subseteq\{0,1,2\}$ be the subset such that $a_jp^{4n_j-m}\mu_j^4\not\equiv 0\mod p$ for all $j\in J$. For any $j\in J$ we have $4n_j+\nu_p(a_j)=2n_3$. In particular $\nu_p(a_j)\equiv \nu_p(a_k)\equiv 2n_3\mod 4$ for all $j,k\in J$.\par
Now if there exists $j\in\{0,1,2\}$ such that $m_j:=4n_j+\nu_p(a_j)\leq 4n_i +\nu_p(a_i)$ for all $i\in\{0,1,2\}$ and $m_j< 2n_3$, then multiplying \eqref{eq: local solutions eq 1} by $p^{-m_j}$ gives
\begin{equation*}
    a_0p^{4n_0-m_j}\mu_0^4+a_1p^{4n_1-m_j}\mu_1^4+a_2p^{4n_2-m_j}\mu_2^4-p^{2n_3-m_j}\mu_3^2=0.
\end{equation*}
Consider this equation modulo $p$. Since $a_jp^{4n_j-m_j}\mu_j^4\not\equiv0\mod p$ there must exist at least one other $k\in\{0,1,2\}$ such that $a_kp^{4n_k-m_j}\mu_k^4\not\equiv 0\mod p$. If exactly one such $j$ exists, then we have
\begin{equation*}
  a_jp^{4n_j-m_j}\mu_j^4\equiv-a_kp^{4n_k-m_j}\mu_k^4\mod p,  
\end{equation*}
and hence by Hensel's Lemma we have
\begin{equation*}
    -\frac{a_k}{a_j}\in\Q_p^{\times 4}.
\end{equation*}
If more than one such $k$ exists, then $a_kp^{4n_k-m_j}\mu_j^4\not\equiv 0\mod p$ for all $k\in \{0,1,2\}$. In particular we have $\nu_p(a_0)\equiv\nu_p(a_1)\equiv \nu_p(a_2)\mod 4$.\par
Now assume \eqref{case: a_0 is a square in p-adics} holds, without loss of generality assume $i=0$. Let $\gamma\in \Q_p^{\times}$ such that $\gamma^2=a_0$, then $[1:0:0:\gamma]\in S(\Q_p)$.\par
Now assume \eqref{case: -a_0/a_1 is a 4th power in p-adics} holds, without loss of generality assume $i=0$ and $j=1$. Let $\gamma\in \Q_p^{\times}$ such that $\gamma^4=\frac{-a_0}{a_1}$, then $[1:\gamma:0:0]\in S(\Q_p)$.\par
Now assume \eqref{case:a_0,a_1 have even class mod 4} holds, then by taking an equivalent surface we may assume the reduction of $S$ modulo $p$ is either
\begin{equation}
\label{eq: case 3 eq 2 for local solubility}
    a_0x_0^4+a_1x_1^4+a_2x_2^4=w^2,
\end{equation}
or the projective cone over the smooth curve given by
\begin{equation}
\label{eq: case 3 eq 1 for local solubility}
    a_1x_1^4+a_2x_2^4=w^2,
\end{equation}
where $a_i\not\equiv 0 \mod p$ for all $i$ in both cases. If the reduction is given by \eqref{eq: case 3 eq 2 for local solubility}, then the hyperplane $\{x_0=0\}$ of $S$ is exactly the smooth curve \eqref{eq: case 3 eq 1 for local solubility}. In particular, it suffices to prove \eqref{eq: case 3 eq 1 for local solubility} has a $\Q_p$-point. Now since \eqref{eq: case 3 eq 1 for local solubility} is a genus $1$ smooth geometrically irreducible curve, by the Hasse-Weil bound we have
\begin{equation*}
    \abs{S(\Fp)}\geq p^\frac{1}{2}(p^\frac{1}{2}-2)+1>0,
\end{equation*}
for $p\neq 2$. Thus by Hensel's Lemma we have $S(\Q_p)\neq \varnothing$.\par
Lastly assume \eqref{case: a_0,a_1,a_3 have same class mod 4} holds. If $\nu_p(a_i)\equiv 0,2 \mod 4$, then the proof follows identically from the argument in case \eqref{case:a_0,a_1 have even class mod 4}. If $\nu_p(a_i)\equiv\pm 1 \mod 4$, then by taking an equivalent surface we may assume the reduction of $S$ modulo $p$ is the projective cone over the smooth curve given by
\begin{equation*}
    a_0x_0^4+a_1x_1^4+a_2x_2^4=0.
\end{equation*}
where $a_i\not\equiv 0 \mod p$. In particular this is a genus $3$ smooth geometrically irreducible curve, and hence by the Hasse-Weil bound we have
\begin{equation*}
    \abs{S(\Fp)}\geq p^\frac{1}{2}(p^\frac{1}{2}-6)+1>0,
\end{equation*}
for $p>33$. Thus by Hensel's Lemma we have $S(\Q_p)\neq \varnothing$.
\end{proof}
\begin{remark}
Note that the condition $p>33$ in Lemma~\ref{lem: local solubility of S} is only required to prove the converse of \eqref{case: a_0,a_1,a_3 have same class mod 4} when $\nu_p(a_i)\neq 0\mod 4$; in all other cases this condition can be replaced by $p\neq 2$.
\end{remark}

We will now assume throughout the rest of this section that \textbf{$S$ is everywhere locally soluble}. Furthermore, for any point $Q\in S(\Q_p)$ (and any prime $p$), we will consider the representation of this point such that all the coordinates are in $\Z_p$ and at least one coordinate is in $\Z_p^{\times}$. We will also let $\{i,j,k\}=\{0,1,2\}$.
\subsubsection{Case $1$: $p\neq 2$ divides exactly one coefficient once}

\begin{lemma}
\label{thm: p odd unqie implies no BM}
If there exists a prime $p\neq 2$ such that $\nu_p(a_i)=1$ and $p\nmid a_ja_k$, then $S$ has no Brauer-Manin obstruction.    
\end{lemma}
\begin{proof}
Let $\mathcal{S}$ be the integral model of $S$ over $\Z_p$ defined by the same equation. Clearly $\mathcal{S}$ is regular, and the reduction modulo $p$ is the projective cone over a smooth curve of genus $1$. By \cite[Thm.~1]{On_the_arithmetic_of_del_Pezzo_surfaces_of_degree_2} we have that $\Br S/\Br\Q$ only has $2$ and $4$-torsion and is at most $2^3$. Since $S$ is a del Pezzo surface, it is well-known that $H^1(S,\O_{S})=0$ and $\Pic S_{\overline{\Q}}$ is torsion-free (see \cite[Lem.~III.3.2.1]{RCurvesOnAVar} and \cite[Prop.~2]{On_the_arithmetic_of_del_Pezzo_surfaces_of_degree_2}, respectively). Therefore by \cite[Thm.~7.4]{BrightMartin2015BrotBadReduction} the result holds for all $p\neq 2$.   
\end{proof} 
\begin{remark}
Consider the surfaces
 \begin{equation*}
     X:103x_0^4+82297x_1^4-47x_2^4=x_3^4,
 \end{equation*}
 and
 \begin{equation*}
     S:103x_0^4+82297x_1^4-47x_2^4=w^2.
 \end{equation*}
 By \cite[Prop.~3.3]{BrightM.J.2011TBoo} we have that $X(\AA_\Q)\neq\varnothing$ (and hence $S(\AA_\Q)\neq \varnothing$) but $X(\AA_\Q)^{\Br}=\varnothing$. However, by Lemma~\ref{thm: p odd unqie implies no BM} we have that $S(\AA_\Q)^{\Br}\neq\varnothing$. In particular having a Brauer-Manin obstruction on a diagonal quartic surface does not imply a Brauer-Manin obstruction on the associated diagonal del Pezzo of degree $2$. 
\end{remark}

\subsubsection{Case $2$: $p$ odd divides exactly one coefficient to an even power}
\begin{definition}
Let the notation be as in Proposition~\ref{prop: quaternion algebra definition from BrightM.J.2011TBoo} and let $a_3:=-1$. We will say $\mathcal{A}$ is \textbf{$p$-normalised} if $a_ny_n\in\Z_p$ for all $n\in\{0,1,2,3\}$ and there exists $m\in\{0,1,2,3\}$ such that $a_my_m\in(\Z_p)^{\times}$.  
\end{definition}

\begin{lemma}
\label{lem: p odd divdes one coff evenly}
Let $p\neq 2$ be a prime such that $\nu_p(a_i)\in 2\Z$ and $p\nmid a_ja_k$. If $\mathcal{A}$ is $p$-normalised, then $\inv_p(\mathcal{A}(-))=0$.
\end{lemma}
\begin{proof}
The result follows by appropriately changing the variables in \cite[Lem.~3.6]{TimSantens}.
\end{proof}

\subsubsection{Case $3$: $p$ odd divides exactly two coefficients to an odd power}
\begin{proposition}
\label{prop: p odd divides a_0 and a_1 oddly and f smooth, then iff fro surjectivity of inv_p}
Let $p\neq 2$ be a prime such that $\nu_p(a_i),\nu_p(a_j)\in \Z\setminus 2\Z$ and $p\nmid a_k$. Assume $\mathcal{A}$ is $p$-normalised. Then $\inv_p(\mathcal{A}(-))$ is a surjection if and only if  
\begin{equation*}
            \left(\frac{\theta_\textbf{a}p^{-\nu_p(\theta_\textbf{a})}}{p}\right)=-1,
\end{equation*}
If $\inv_p(\mathcal{A}(-))$ is not a surjection, then $\inv_p(\mathcal{A}(-))=0$. 
\end{proposition}
\begin{proof}
Without loss of generality assume $p$ divides $a_0$ and $a_1$. Firstly if $\left(\frac{\theta_\textbf{a}p^{-\nu_p(\theta_\textbf{a})}}{p}\right)=1$, then by the properties of the Hilbert symbol we have $\inva{Q}=0$ for all $Q\in S(\Q_p)$. Thus assume $\left(\frac{\theta_\textbf{a}p^{-\nu_p(\theta_\textbf{a})}}{p}\right)=-1$, then since $f\neq 0$ modulo $p$ we must have that $p$ does not divide $y_2$ and $y_3$. Thus we have
\begin{equation*}
    a_2y_2^2\equiv y_3^2\mod p.
\end{equation*}
Therefore by Hensel's Lemma there exists $\gamma\in \Z_p^{\times}$ such that $\gamma^2=a_2$ and $y_2\gamma\equiv y_3\mod p$, and hence we have points $Q_{\pm}=[0:0:1:\pm\gamma]\in S(\Q_p)$. Thus we have
\begin{align*}
    -a_1f(Q_+)f(Q_-)=-a_1(a_2^2y_2^2-a_2y_3^2)=a_0a_1a_2y_0^2+(a_1\gamma y_1)^2\in \operatorname{N}_{\Q_p(\theta_\textbf{a})/\Q_p}\Q_p(\theta_\textbf{a})^{\times}.
\end{align*}
Therefore by the properties of the Hilbert symbol we have
\begin{equation*}
    \inva{Q_+}=\inva{Q_-}+\inv_p((\theta_\textbf{a},-a_1))=\inva{Q_-}+\frac{1}{2}.\qedhere
\end{equation*}
\end{proof}
We now split into the subcases $p\equiv\pm 1 \mod 4$:
\begin{proposition}
\label{thm: p = 3 mod 4 and divdes two coefccients oddly implies surjective inv_p}
Let $p\equiv3\mod 4$ be a prime such that $\nu_p(a_i)=\nu_p(a_j)\in \Z\setminus 2\Z$ and $p\nmid a_k$. Then $\inv_p(\mathcal{A}(-))$ is a surjective if and only if 
\begin{equation*}
    \left(\frac{-a_ia_jp^{-\nu_p(a_ia_j)}}{p}\right)=-1.
\end{equation*}
\end{proposition}
\begin{proof}
 Without loss of generality assume $p$ divides $a_0$ and $a_1$ such that $\nu_p(a_0)\in\{1,3\}$. If $\left(\frac{a_2}{p}\right)=1$, then after possibly changing rational point $P$ so that $\mathcal{A}$ is $p$-normalised, the result follows by Proposition~\ref{prop: p odd divides a_0 and a_1 oddly and f smooth, then iff fro surjectivity of inv_p}. Thus assume $\left(\frac{a_2}{p}\right)=-1$. Since $S(\Q_p)\neq \varnothing$, by Lemma~\ref{lem: local solubility of S} we must have $\left(\frac{-a_0a_jp^{-\nu_p(a_ia_j)}}{p}\right)=1$. Thus it suffices to prove if $\left(\frac{a_2}{p}\right)=-1$, then $\inv_p(\mathcal{A}(-))$ is constant.\par
 By Lemma~\ref{lem: local solubility of S} there exists $\gamma\in \Z_p^{\times}$ such that $\gamma^4= -\frac{a_0}{a_1}$. Let $R_\pm=[1:\pm \gamma:0:0]\in S(\Q_p)$. For any $Q=[x_0:x_1:x_2:w]\in S(\Q_p)$, we must have that $p$ divides $x_2$ and $w$, and hence we have
\begin{equation*}
    a_0p^{-\nu_p(a_0)}x_0^4\equiv -a_1p^{-\nu_p(a_1)}x_1^4 \mod p,
\end{equation*}
 where $p$ does not divide $x_0$ or $x_1$. In particular, the reduction modulo $p$ of any $Q\in S(\Q_p)$, denoted $\Tilde{Q}$, is contained in the set $\{R_\pm\}$.  Let $\mathcal{S}$ be the integral model of $S$ over $\Z_p$ defined by the same equation, then we have $\mathcal{S}^{sm}(\Z_p)=S(\Z_p)$. Therefore by \cite[Prop.~5.1]{BrightMartin2015BrotBadReduction} we have
 \begin{equation*}
     \inva{Q}\in\{\inva{R_-},\inva{R_+}\},
 \end{equation*}
 for any $Q\in S(\Z_p)$. In particular since $f(R_+)=f(R_-)$, the invariant map is constant.
\end{proof}
\begin{proposition}
\label{thm: p odd divides two coefficents to the same odd power implies inv_p surjective}
 Let $p\equiv1\mod 4$ be a prime such that $\nu_p(a_i)=\nu_p(a_j)\in \Z\setminus 2\Z$ and $p\nmid a_k$. Then $\inv_p(\mathcal{A}(-))$ is constant if and only if 
 \begin{equation*}
     \left(\frac{-a_ia_jp^{-\nu_p(a_ia_j)}}{p}\right)=\left(\frac{a_k}{p}\right)=1.
 \end{equation*}
\end{proposition}
\begin{proof}
Without loss of generality assume $i=0$, $j=1$ and $k=2$. Firstly if $\left(\frac{a_2}{p}\right)=1$, then after possibly changing rational point $P$ so that $\mathcal{A}$ is $p$-normalised, the result follows by Proposition~\ref{prop: p odd divides a_0 and a_1 oddly and f smooth, then iff fro surjectivity of inv_p}. Therefore assume $\left(\frac{a_2}{p}\right)=-1$, then by Lemma~\ref{lem: local solubility of S} we must have that $-\frac{a_0}{a_1}\in \Q_p^{\times4}$. Let $\gamma,\delta \in \Z_p^{\times}$ such that $\gamma^4=-\frac{a_0}{a_1}$ and $\delta^2=-1$, then we have four points $R_{\pm}=[1:\pm\gamma:0:0],R'_{\pm}=[1:\pm \delta\gamma:0:0]\in S(\Q_p)$. Furthermore, we have
\begin{align*}
    -a_1f(R_+)f(R'_+)=-a_1(a_0^2y_0^2+a_0a_1y_1^2)=a_0a_1a_2y_2^2-a_0a_1y_3^2\in \operatorname{N}_{\Q_p(\theta_\textbf{a})/\Q_p}\Q_p(\theta_\textbf{a})^{\times}.
\end{align*}
Thus we have
\begin{equation*}
    \inva{R_+}=\inva{R'_+}+\inv_p((\theta_\textbf{a},-a_1))=\inva{R'_+}+\frac{1}{2}.\qedhere
\end{equation*}
\end{proof}

\subsubsection{Changing $\mathcal{A}_\textbf{a}$ by $\textbf{u}\in(\Q^{\times})^{3}$}
In this section we analyse how the local invariant map changes when the surface $S_\textbf{a}$ changes.\\
Let $\textbf{u}=(u_0,u_1,u_2)\in(\Q^{\times})^{3}$ and let $\textbf{au}^2=(a_0u_0^2,a_1u_1^2,a_2u_2^2)$. By repeating the construction of $\mathcal{A}_\textbf{a}$ given in Proposition~\ref{prop: quaternion algebra definition from BrightM.J.2011TBoo} for $\textbf{au}^2$ we have
\begin{equation*}
    \mathcal{A}_{\textbf{au}^2}=\left(\theta_\textbf{a},\frac{u_0a_0y_0x_0^2+u_1a_1y_1x_1^2+u_2a_2y_2x_2^2-y_3w}{x_0^2}\right)\in \Br S_{\textbf{au}^2}.
\end{equation*}
The similarity between $\mathcal{A}_\textbf{a}$ and $\mathcal{A}_{\textbf{au}^2}$ allows us to express the associated invariant maps in terms of each other in certain cases (which we will use in \S\ref{sec: counting}):

\begin{lemma}
\label{lem: formula for when p divdies w_ij}
Let $p\equiv 3 \mod 4$, $\theta_\textbf{a}\not\in\Q_p^{\times2}$, $\nu_p(a_i)=\nu_p(a_{j})\in\Z\setminus 2\Z$ and $p\nmid a_k$. Assume that $-\frac{a_i}{a_j}\in \Q_p^{\times 4}$. Then $\inv_p(\mathcal{A}_{\textbf{a}}(-))$ is constant and for all $\textbf{u}\in(\Z_p^{\times})^{3}\cap (\Q^{\times})^{3}$ we have
\begin{equation}
\label{eq: invarian maps formula at p=3mod4}
    \inv_p(\mathcal{A}_{\textbf{au}^2}(-))=\inv_p(\mathcal{A}_{\textbf{a}}(-)) +\frac{\left(\frac{u_iu_{j}}{p}\right)-1}{4}.
\end{equation}
\end{lemma}
\begin{proof}
Without loss of generality assume $i=0$ and $j=1$, then by Proposition~\ref{thm: p = 3 mod 4 and divdes two coefccients oddly implies surjective inv_p} it suffices to prove that \eqref{eq: invarian maps formula at p=3mod4} holds.\par
Let $\mathcal{A}'_\textbf{a}$ and $\mathcal{A}'_{\textbf{au}^2}$ denote the algebras $\mathcal{A}$ and $\mathcal{A}_{\textbf{u}}$ in \cite[Prop.~3.8]{TimSantens}, then by \cite[Prop.~3.8]{TimSantens} we have
\begin{align*}
    \inv_p(\mathcal{A}_{\textbf{au}^2}(\pi_{\textbf{au}^2}(-)))&=\inv_p(\mathcal{A}'_{\textbf{au}^2}(-))\\
    &=\inv_p(\mathcal{A}'_{\textbf{a}}(-)) +\frac{\left(\frac{u_iu_{j}}{p}\right)-1}{4}\\
    &=\inv_p(\mathcal{A}_{\textbf{a}}(\pi_{\textbf{au}^2}(-))) +\frac{\left(\frac{u_iu_{j}}{p}\right)-1}{4}.
\end{align*}
In particular since both invariant maps are constant \eqref{eq: invarian maps formula at p=3mod4} holds.
\end{proof}

\begin{lemma}
\label{lem: formula for moving between A and A_u}   
Let $p$ be any place of $\QQ$. Let $\textbf{u}\in (\Q^{\times})^{3}$, and let $\textbf{v}\in (\Z_p^{\times})^{3}$ such that $u_n=v_n^2$ for all $n$. Then
\begin{equation*}
    \inv_p(\mathcal{A}_\textbf{a}(-))=\inv_p(\mathcal{A}_{\textbf{au}^2}(\phi(-))),
\end{equation*}
where $\phi:S_\textbf{a}(\Z_p) \rightarrow S_{\textbf{au}^2}(\Z_p)$ is the bijection given by $[x_0:x_1:x_2:w]\mapsto [x_0v_0^{-1}:x_1v_1^{-1}:x_2v_2^{-1}:w]$.
\end{lemma}
\begin{proof}
The result is immediate.
\end{proof}

\begin{lemma}
\label{lem: A is p-normalised if and only if A_u is p-normalised}   
Let $a_0,a_1,a_2,u_i,u_j\in \Z_p^{\times}$ and $u_k\in \Z_p$. Let $\textbf{u}=(u_0,u_1,u_2)$, then $\mathcal{A}_\textbf{a}$ is $p$-normalised if and only if $\mathcal{A}_{\textbf{au}^2}$ is $p$-normalised.
\end{lemma}
\begin{proof}
Without loss of generality we may assume $u_0\in \Z_p$. If $\mathcal{A}_\textbf{a}$ is $p$-normalised and $\mathcal{A}_{\textbf{au}^2}$ is not $p$-normalised, then $p$ divides $y_1, y_2$ and $y_3$, but not $y_0$. However we have
\begin{equation*}
    0=\nu_p(a_0y_0^2)=\nu_p(-a_1y_1^2-a_2y_2^2+y_3^2)\geq 2,
\end{equation*}
a contradiction. If $\mathcal{A}_{\textbf{au}^2}$ is $p$-normalised, then $p$ does not divide either $u_0y_0, y_1, y_2$ or $y_3$. Therefore $\mathcal{A}_\textbf{a}$ must also be $p$-normalised. 
\end{proof}

\section{Counting}
\label{sec: counting}
\subsection{Set-up}
Throughout this section we will closely follow \S5 in \cite{TimSantens}. To make the similarities between the papers clearer we will use the same notation as used in \cite{TimSantens} when appropriate to do so. Furthermore, unless explicitly stated otherwise, we let $\{i,j,k\}=\{0,1,2\}$. We will write $\prod_i$ to mean the product over all $i\in\{0,1,2\}$. We will write $\prod_{i<j}$ to mean the product over all $i$ and $j$ in $\{0,1,2\}$ such that $i$ is less than $j$. Furthermore, we will write $\prod_{i\neq j}$ to mean the product over all $i$ in $\{0,1,2\}$ such that $i$ is not equal to $j$.\par
Let $S$ be the set of all primes less than $98$, and let
\begin{equation*}
    \textbf{$\Phi$}(T):= \left\{\textbf{A}=(A_0,A_1,A_2)\in (\Z\setminus\{0\})^3: \begin{array}{l}
       \abs{A_i}\leq T \text{ for all } i,\\
         p\mid A_1A_2A_3 \implies \left[ p \in S \text{ or } p^3\mid A_1A_2A_3 \right]  
    \end{array} \right\}.
\end{equation*}
For \textbf{A} $\in$ \textbf{$\Phi$}$(T)$ we define $m_\textbf{A}:= \operatorname{rad}(A_0A_1A_2\prod_{p\in S}p)$ and $\theta_\textbf{A}:=-A_0A_1A_2$. Let \textbf{$\Omega$} denote the set of all $3$-tuples of cosets of $\left(\Z/8m_\textbf{A}\Z\right)^{\times4}$ in $\left(\Z/8m_\textbf{A}\Z\right)^{\times}$, and let
\begin{equation*}
    \Psi_\textbf{A}:= \{\textbf{M}=(M_0,M_1,M_2)\in \textbf{$\Omega$}: M_0M_1M_2=\gamma^2\left(\Z/8m_\textbf{A}\Z\right)^{\times4}, \text{ for some } \gamma \in \left(\Z/8m_\textbf{A}\Z\right)^{\times}\}.
\end{equation*}
For any choice of $(\textbf{A}, \textbf{M})\in$  \textbf{$\Phi$}$(T)\times\Psi_\textbf{A}$ and $T\in \RR_{\geq0}$, we define
\begin{align*}
N_{\textbf{A},\textbf{M}}(T):= \left\{ \begin{array}{l} (\textbf{u},\textbf{v},\textbf{w})\\
\textbf{u}=(u_0,u_1,u_2)\in \NN^3\\
\textbf{v}=(v_{01},v_{02},v_{12})\in \NN^3\\
\textbf{w}=(w_{01},w_{02},w_{12})\in \NN^3
\end{array} : \begin{array}{l}
    \abs{A_iu_i^2\prod_{j\neq i}v_{ij}w_{ij}}\leq T, \text{ for all }i,     \\
    \mu\left(m_{\textbf{A}}\prod_iu_i\prod_{j\neq i}v_{ij}w_{ij}\right)^2=1,\\
    p\mid v_{ij} \implies \theta_\textbf{A} \in \Q_p^{\times2}, p\mid w_{ij} \implies \theta_\textbf{A} \notin \Q_p^{\times2},\\
     u_i^2\prod_{j\neq i}v_{ij}w_{ij} (\mod 8m_\textbf{A}) \in M_i \text{ for all } i,\\
     S_{\textbf{a}}(\Q_p)\neq \varnothing \text{ for } p\nmid m_\textbf{A}
  \end{array}\right\}.
\end{align*}
Now we define $a_i:=A_iu_i^2\prod_{j\neq i}v_{ij}w_{ij}$ and let $\textbf{a}:=(a_0,a_1,a_2)$; these $a_i$ will be the coefficients of the surfaces $S_{\textbf{a}}$ that we will be counting. We define $t_{ij}:=v_{ij}w_{ij}$. Lastly, let
   \begin{equation*}
       M_\textbf{A}:=\begin{cases}
           \frac{1}{2}, &\text{if } \theta_\textbf{A}\in-\Q^{\times 2},\\
           \frac{3}{8}, &\text{if } \theta_\textbf{A}\not\in \Q^{\times 2}.
       \end{cases}
   \end{equation*} 
This will show up in the power of the $\log T$ term of the asymptotic formulas given in Theorem~\ref{thm: main theorem split into squares and non squares}.
\begin{remark}
\label{rem: a completely determines A, M. u, v, w}
 If we are given a tuple $\textbf{a}$ such that $a_i:=A_iu_i^2\prod_{j\neq i}v_{ij}w_{ij}$ for all $i$, where $(\textbf{u},\textbf{v},\textbf{w})\in \NAM(T)$ for some $\textbf{A}, \textbf{M}$ and $T$, then the conditions on $\NAM(T)$ uniquely determine \textbf{A}, \textbf{M}, \textbf{u}, \textbf{v} and \textbf{w}. In particular for any surface $S_{\textbf{a}}$ there is at most one element in one set $N_{\textbf{A},\textbf{M}}(T)$ corresponding to this surface. \par
 Similarly, given $t_{ij}:=v_{ij}w_{ij}$ for some $(\textbf{u},\textbf{v},\textbf{w})\in \NAM(T)$, the conditions on $\NAM(T)$ uniquely determine $v_{ij}$ and $w_{ij}$. In particular if we write $\textbf{t}=(t_{01},t_{02},t_{12})$, then for each $(\textbf{u},\textbf{v},\textbf{w})\in \NAM(T)$ there exists a unique tuple $(\textbf{u},\textbf{t})$ associated to it.  
\end{remark}
\begin{remark}
If $\abs{A_i}>T$ for any $i$, then clearly $N_{\textbf{A},\textbf{M}}(T)=\varnothing$ (hence our choice of \textbf{$\Phi$}$(T)$).   
\end{remark}
\begin{remark}
For any $(\textbf{u},\textbf{v},\textbf{w})\in N_{\textbf{A},\textbf{M}}(T)$ we must have $\prod_iu_i^2\prod_{j\neq i}v_{ij}^2w_{ij}^2\in M_0M_1M_2$. Therefore if $M_0M_1M_2$ is not represented by a square element of $\left(\Z/8m_\textbf{A}\Z\right)^{\times}$, then $N_{\textbf{A},\textbf{M}}(T)=\varnothing$ (hence our choice of $\Psi_\textbf{A}$).     
\end{remark}

\subsection{Reductions}
\subsubsection{Reducing $N^{\Br}_{\neq\pm\Box}(T)$ to $\NAM^{\Br}(T)$}
In this section we will prove that to find asymptotic formulae for $\#N^{\Br}_{=-\Box}(T)$ and $\#N^{\Br}_{\neq\pm\Box}(T)$ it suffices to find an asymptotic formula for another set (namely $\#\NAM^{\Br}(T)$).\\
Let 
\begin{equation*}
    N^{\Br}_{\textbf{A}, \textbf{M}}(T):=\{(\textbf{u},\textbf{v},\textbf{w})\in N_{\textbf{A},\textbf{M}}(T): S_{\textbf{a}} \text{ has a Brauer-Manin obstruction}\},
\end{equation*}
then we have the following reduction:
\begin{lemma}
\label{lem: reducing N^Br to sum of N^Br_A,M}
We have
\begin{equation*}
    \#N^{\Br}_{=\Box}(T)=\mathop{\sum}_{\substack{(\textbf{A},\textbf{M})\in \textbf{$\Phi(T)$}\times\Psi_\textbf{A}\\ \theta_\textbf{A} \in (\Q^{\times})^2}}  \#\NAM^{\Br}(T),
\end{equation*}
\begin{equation*}
    \#N^{\Br}_{=-\Box}(T)=\mathop{\sum}_{\substack{(\textbf{A},\textbf{M})\in \textbf{$\Phi(T)$}\times\Psi_\textbf{A}\\ \theta_\textbf{A} \in -(\Q^{\times})^2}}  \#\NAM^{\Br}(T),
\end{equation*}
\begin{equation*}
    \#N^{\Br}_{\neq\pm\Box}(T)=\mathop{\sum}_{\substack{(\textbf{A},\textbf{M})\in \textbf{$\Phi(T)$}\times\Psi_\textbf{A}\\ \theta_\textbf{A} \notin\pm (\Q^{\times})^2}}  \#\NAM^{\Br}(T).
\end{equation*}
\end{lemma}
\begin{proof}
 As in the above exposition, given $(\textbf{u},\textbf{v},\textbf{w})\in N_{\textbf{A},\textbf{M}}^{\Br}(T)$, for some $(\textbf{A},\textbf{M})\in$ \textbf{$\Phi(T)$}$\times\Psi_\textbf{A}$, we obtain a tuple $\textbf{a}=(a_0,a_1,a_2)$ associated to it. In particular we can indeed view the right-hand side as a subset of the left-hand side. Furthermore as in Remark~\ref{rem: a completely determines A, M. u, v, w}, any such tuple $\textbf{a}=(a_0,a_1,a_2)$ uniquely determines \textbf{A}, \textbf{M}, \textbf{u}, \textbf{v} and \textbf{w}, and hence each set on the right-hand side is disjoint from each other. Lastly, the only tuples \textbf{a} that may be counted in the left-hand side but are not counted in the right-hand side are those tuples such that $S_\textbf{a}$ is everywhere locally soluble and there exists a prime $p\not\in S$ such that $\nu_p(a_0a_1a_2)=1$. However by Lemma~\ref{thm: p odd unqie implies no BM} such surfaces have no Brauer-Manin obstruction. In particular the left-hand side does not count these tuples either.
\end{proof}

\begin{lemma}
\label{lem: theta_a a square bounded}
If $\theta_\textbf{A}\in \Q^{\times 2}$, then for $T>2$ we have
\begin{equation*}
    \#\NAM^{\Br}(T)\ll \frac{T^\frac{3}{2}(\log T)}{\abs{\theta_\textbf{A}}^{\frac{1}{2}}}.
\end{equation*}
\end{lemma}
\begin{proof}
Let $\mathcal{B}_{ij3,\sqrt{\theta_\textbf{A}}}$ be as in Proposition~\ref{prop: quaternion algebra definition from BrightM.J.2011TBoo}, then from Remark~\ref{rem: pullback of algbra is algebra from X} we have $\pi_\textbf{a}^*\mathcal{B}_{ij3,\sqrt{\theta_\textbf{A}}}= \mathcal{B}_{ij3,\sqrt{\theta_\textbf{A}}}'$, where $\mathcal{B}_{ij3,\sqrt{\theta_\textbf{A}}}'$ is the algebra defined in \cite[Prop.~3.1.(ii)]{TimSantens}. In particular if we let $p>97$ be prime, $a_n,a_na_m\neq \pm1,\pm2 \in \Q^{\times}/\Q^{\times 2}$ for all $n\neq m\in\{0,1,2\}$, $\nu_p(a_k)\equiv 2\mod 4$ and $p\nmid a_ia_j$, then by \cite[Lem.~3.5]{TimSantens} we have that $\inv_p(\mathcal{B}_{ij3,\sqrt{\theta_\textbf{A}}}'(-))$ is surjective. Thus $\inv_p(\mathcal{B}_{ij3,\sqrt{\theta_\textbf{A}}}(-))$ is surjective, and hence induces no Brauer-Manin obstruction. Furthermore from the magma code used in the proof of \cite[Thm.~2]{On_the_arithmetic_of_del_Pezzo_surfaces_of_degree_2} (which can be found in \cite{magmacode}) we know that under these hypothesis $\Br S_\textbf{a}/\Br \Q \cong \Z/2\Z$, and hence it is generated by $\mathcal{B}_{ij3,\sqrt{\theta_\textbf{A}}}$. Therefore there is no Brauer-Manin obstruction in this case. Thus we can bound $\NAM^{\Br}(T)$ by the three subsets of $\NAM(T)$ defined by
\begin{enumerate}
    \item $a_ia_j= \pm1,\pm2 \in \Q^{\times}/\Q^{\times 2}$, for some $i\neq j$;
    \item $a_k= \pm1,\pm2 \in \Q^{\times}/\Q^{\times 2}$, for some $k$;
    \item $u_i=1$ for all $i$.
\end{enumerate}
Consider the subset of $\NAM(T)$ such that $a_ia_j=\pm1,\pm2 \in \Q^{\times}/\Q^{\times 2}$ for some $i\neq j$. We may assume $i=0$ and $j=1$. Suppose $t_{k\ell}\neq 1$ for some $(k,\ell)\neq (0,1)$, then there exists $p \not\in S$ prime such that $p\mid t_{k\ell}$. However $\nu_p(a_0a_1)=1$, which contradicts our assumption. Thus $t_{k\ell}=1$ for all $(k,\ell)\neq (0,1)$, and hence we can bound this set by
\begin{equation*}
    \ll \mathop{\sum\sum}_{\substack{t_{01},u_k\\ \abs{A_k}u_k^2\prod_{i\neq k}t_{ik}\leq T}}1\ll \frac{T^\frac{3}{2}}{\abs{\theta_\textbf{A}}^\frac{1}{2}}\mathop{\sum}_{\substack{t_{01}\leq T}} \frac{1}{t_{01}}\ll \frac{T^\frac{3}{2}\log T}{\abs{\theta_\textbf{A}}^\frac{1}{2}}.
\end{equation*}
Now consider the subset of $\NAM^{\Br}(T)$ such that $a_i=\pm1,\pm2 \in \Q^{\times}/\Q^{\times 2}$ for some $i$. We may assume $i=0$. Suppose $t_{k\ell}\neq 1$ for some $(k,\ell)\neq (1,2)$, then there exists $p \not\in S$ prime such that $p\mid t_{k\ell}$. However $\nu_p(a_0)=1$, which contradicts our assumption. Thus $t_{k\ell}\neq 1$ for all $(k,\ell)\neq (1,2)$, and hence we can bound this set by
\begin{equation*}
    \ll \mathop{\sum\sum}_{\substack{t_{12},u_k\\ \abs{A_k}u_k^2\prod_{i\neq k}t_{ik}\leq T}}1\ll \frac{T^\frac{3}{2}}{\abs{\theta_\textbf{A}}^\frac{1}{2}}\mathop{\sum}_{\substack{t_{12}\leq T}} \frac{1}{t_{12}}\ll \frac{T^\frac{3}{2}\log T}{\abs{\theta_\textbf{A}}^\frac{1}{2}}.
\end{equation*}
For the third set, by \cite[Lem.~4.17]{TimSantens} we can bound this set by
\begin{align*}
    \ll\mathop{\sum}_{\substack{t_{ij}\\\abs{A_i}\prod_{i\neq j}t_{ij}\leq T}}1&\ll\int^{\abs{A_i}\prod_{i\neq j}t_{ij}\leq T}_{t_{ij}\geq \frac{1}{2}}\prod_{i<j}dt_{ij}.
\end{align*}
Thus we have
\begin{align*}
    \int^{\abs{A_i}\prod_{i\neq j}t_{ij}\leq T}_{t_{ij}\geq \frac{1}{2}}\prod_{i<j}dt_{ij}&\ll \frac{T}{\abs{\theta_\textbf{A}}}\int^{t_{02}t_{12}\leq T}_{t_{02},t_{12}\geq \frac{1}{2}}\frac{1}{t_{02}t_{12}}dt_{02}dt_{12}\\
    &\ll \frac{T (\log T)^2}{\abs{\theta_\textbf{A}}}\\
    &\ll\frac{T^\frac{3}{2}\log T}{\abs{\theta_\textbf{A}}^\frac{1}{2}}.\qedhere
\end{align*}
\end{proof}
Since we have now dealt with the case $\theta_\textbf{A}\in \Q^{\times 2}$, \textbf{we will assume $\theta_\textbf{A}\not\in \Q^{\times 2}$ for the rest of the paper, except in the proof of Theorem~\ref{thm: main theorem split into squares and non squares}}.
\subsubsection{Reducing $\NAM^{\Br}(T)$ to $\NAM^{\mathcal{A}}(T)$}
In this subsection we will prove that $\mathcal{A}_\textbf{a}$ induces the Brauer-Manin obstruction on the majority of surfaces in $\NAM^{\Br}(T)$. To do this we need the following result, which we will prove in \S\ref{subsec: The error terms}:
\begin{definition}
   \begin{equation*}
       \NAM^{loc}(T):=\left\{ (\textbf{u},\textbf{v},\textbf{w})\in \NAM(T) :  p \mid w_{ij} \implies \left[ p\equiv 3\mod 4 \text{ and } -\frac{a_i}{a_{j}}\in\Z_p^{\times 2}\right]\right\}.
   \end{equation*}
      \begin{equation*}
       \NAM^{\mathcal{A}}(T):=\left\{ (\textbf{u},\textbf{v},\textbf{w})\in \NAM(T) :  S_{\textbf{a}} \text{ has a Brauer-Manin obstruction induced by } \mathcal{A}_\textbf{a}\right\}.
   \end{equation*}
\end{definition}
\begin{lemma}
\label{lem: bound for downward open sets for reduction}
For all $\epsilon,C>0$, $T>2$ and $k\in\{0,1,2\}$ we have
\begin{equation}
\label{eq: first downward open set}
  \#\left\{(\textbf{u},\textbf{v},\textbf{w})\in \NAM^{loc}(T) : w_{ij}\leq (\log T^\frac{3}{2})^C \text{ for all } \{i,j\}\right\} \ll_{\epsilon,C} \frac{T^\frac{3}{2}(\log T)^{\frac{5}{2}M_\textbf{A}}}{\abs{\theta_\textbf{A}}^{\frac{1}{2}-\epsilon}},  
\end{equation}
\begin{equation}
\label{eq: second downward open set}
    \#\left\{(\textbf{u},\textbf{v},\textbf{w})\in \NAM^{loc}(T) : u_i\leq (\log T^\frac{3}{2})^C \text{ for all } i\neq k \right\} \ll_{\epsilon,C} \frac{T^\frac{3}{2}(\log T)^{\frac{5}{2}M_\textbf{A}}}{\abs{\theta_\textbf{A}}^{\frac{1}{2}-\epsilon}}.
\end{equation}
\end{lemma}

\begin{lemma}
\label{lem: reducing N^Br to N^A}
For $T>2$ we have
 \begin{equation*}
     \#N^{\Br}_{\textbf{A},\textbf{M}}(T)=\#N^{\mathcal{A}}_{\textbf{A},\textbf{M}}(T)+O_{\epsilon}\left(\frac{T^\frac{3}{2}(\log T)^{\frac{5}{2}M_\textbf{A}}}{\abs{\theta_\textbf{A}}^{\frac{1}{2}-\epsilon}}\right).
 \end{equation*}
\end{lemma}
\begin{proof}
Firstly we can bound the number of surfaces which have a Brauer-Manin obstruction not induced by $\mathcal{A}_\textbf{a}$ by the number of surfaces for which either $\abs{\Br S_\textbf{a}/\Br \Q}>2$ or $\mathcal{A}_\textbf{a}\in \Br \Q$. For the first case, from the magma code used in the proof of \cite[Thm.~2]{On_the_arithmetic_of_del_Pezzo_surfaces_of_degree_2} (which can be found in \cite{magmacode}) it can be checked explicitly that in every case that $\abs{\Br S_\textbf{a}/\Br \Q}>2$ we have that $a_i$ or $a_ia_j=\pm1,\pm2 \in \Q^{\times}/\Q^{\times 2}$ for some $i\neq j$. Therefore the set in the first case can be bounded by the union of the two sets in $\NAM^{\Br}(T)$ defined by these two conditions. Thus as in the proof of Lemma~\ref{lem: theta_a a square bounded} both sets can be bounded by
\begin{equation*}
    \ll \frac{T^\frac{3}{2}\log T}{\abs{\theta_\textbf{A}}^\frac{1}{2}}.
\end{equation*}
For the second case, $\mathcal{A}'_\textbf{a}\in \Br X_\textbf{a}$ denote the element constructed in \cite[Prop.~3.1.(i)]{TimSantens}. Recall we have $\pi^*\mathcal{A}_\textbf{a}=\mathcal{A}'_\textbf{a}$ (see Remark~\ref{rem: pullback of algbra is algebra from X}). Thus if $\mathcal{A}_\textbf{a}\in \Br\Q$, then $\mathcal{A}'_\textbf{a}\in\Br\Q$. In particular we can bound this set by the number of surfaces $X_\textbf{a}$ such that $\mathcal{A}'_\textbf{a}\in\Br\Q$. By \cite[Prop.~3.7]{TimSantens} this set is bounded by the set of tuples in $\NAM(T)$ such that $w_{ij}=1$ for all $i,j$. This is equal to the subset on the left-hand side of \eqref{eq: first downward open set} and hence by Lemma~\ref{lem: bound for downward open sets for reduction} the result holds.
\end{proof}
\subsubsection{Reducing $\NAM^{\mathcal{A}}(T)$ to $\NAM^{loc}(T)$}
In this section we reduce counting $\#\NAM^{\mathcal{A}}(T)$ to counting $\#\NAM^{loc}(T)$.\par
Fix a choice of $(\textbf{u},\textbf{v},\textbf{w})\in \NAM(T)$ and consider the following condition:
\begin{equation}
\label{eq: condtion for invariant being constant}
    \text{If }p\in S\cup\{p:p\mid \theta_\textbf{A}\}\cup\{\infty\}, \text{ then } S_{\textbf{a}}(\Q_p)\neq \varnothing \text{ and } \inva{-} \text{ is constant.}
\end{equation}
\begin{lemma}
\label{lem: condtion is independant of choice of coefficents}
 Condition~\eqref{eq: condtion for invariant being constant} is independent of choice of $(\textbf{u},\textbf{v},\textbf{w})\in \NAM(T)$. 
\end{lemma}
\begin{proof}
Let $(\textbf{u},\textbf{v},\textbf{w})$ and $(\textbf{u}',\textbf{v}',\textbf{w}')$ be two different elements in $\NAM(T)$. Let $\textbf{a}$ and $\textbf{a}'$ denote the coefficients associated to these elements, respectively, then it suffices to prove $S_{\textbf{a}}\cong S_{\textbf{a}'}$ over $\Q_p$ for all $p\in S\cup\{p:p\mid \theta_\textbf{A}\}\cup\{\infty\}$.\\
For $i\in\{0,1,2\}$ we have
\begin{equation}
\label{eq: fraction of coefficents 4-th power mod p}
    \frac{a_i}{a_i'}=\frac{u_i\prod_{j\neq i}v_{ij}w_{ij}}{u_i'\prod_{j\neq i}v_{ij}'w_{ij}'}\in M_i(M_i)^{-1}=(\Z/8m_\textbf{A}\Z)^{\times4}.
\end{equation}
In particular by Hensel's Lemma we have
\begin{equation*}
    \frac{a_i}{a_i'}\in \Q_p^{\times4},
\end{equation*}
for all $p\neq 2\in S\cup\{p:p\mid \theta_\textbf{A}\}$. Thus we clearly have $S_{\textbf{a}}\cong S_{\textbf{a'}}$ over $\Q_p$ for all such $p$.\par
If $p=2$, then by Hensel's Lemma we have an isomorphism $(\Z/16\Z)^{\times}\xrightarrow{\sim}(\Z_2^{\times})/(\Z_2)^{\times4}$. Therefore by \eqref{eq: fraction of coefficents 4-th power mod p} (and since $2\mid m_\textbf{A}$) we have
\begin{equation*}
    \frac{a_i}{a_i'}\in (\Z_2)^{\times4}.
\end{equation*}
Thus we have $S_{\textbf{a}}\cong S_{\textbf{a'}}$ over $\Q_2$.\par 
If $p=\infty$, then since every element of $\NAM(T)$ is a tuple of strictly positive integers, a $4$-th root of every component of the tuple exist in $\RR$. In particular we have that $S_{\textbf{a}}\cong S_{\textbf{a'}}$ over $\RR$.
\end{proof}
In particular by Lemma~\ref{lem: condtion is independant of choice of coefficents} we know that condition~\eqref{eq: condtion for invariant being constant} only depends on the choice of $\textbf{A}$ and $\textbf{M}$, and hence the following function is well-defined:
\begin{definition}
For $(\textbf{A},\textbf{M})\in$ \textbf{$\Phi(T)$}$\times\Psi_\textbf{A}$ define 
\begin{equation*}
    \eta(\textbf{A},\textbf{M}):=\begin{cases}1, & \text{ if condition~\eqref{eq: condtion for invariant being constant} is satisfied},\\
    0, & \text{ otherwise.}
    \end{cases}
\end{equation*}
\end{definition}
We now have the following reduction:
\begin{lemma}
\label{lem: reducing N^A to N^loc}
For $T>2$ we have
 \begin{equation}
     \#\NAM^{\mathcal{A}}(T)=\frac{\eta(\textbf{A},\textbf{M})}{2}\#\NAM^{loc}(T)+O_{\epsilon}\left(\frac{T^\frac{3}{2} (\log T)^{\frac{5}{2}M_\textbf{A}}}{\abs{\theta_\textbf{A}}^{\frac{1}{2}-\epsilon}}\right).
 \end{equation}
\end{lemma}
\begin{proof}
We closely follow the proof of \cite[Lem.~5.6]{TimSantens}: Firstly if $\eta(\textbf{A},\textbf{M})=0$, then clearly $\#\NAM^{\mathcal{A}}(T)=0$.\par
If $\eta(\textbf{A},\textbf{M})=1$, then observe by Proposition~\ref{thm: p = 3 mod 4 and divdes two coefccients oddly implies surjective inv_p} and Proposition~\ref{thm: p odd divides two coefficents to the same odd power implies inv_p surjective} we have $\NAM^{\mathcal{A}}(T)\subseteq  \NAM^{loc}(T)$. Let $(\textbf{u},\textbf{v},\textbf{w})\in \NAM^{loc}(T)$ and let $S_{\textbf{a}}$ be the associated surface. We will construct a new set, denoted $\NAM^{loc'}(T)$, which is bijective to $\NAM^{loc}(T)$ in order to remove the local solubility condition in $\NAM(T)$, and hence remove the dependence between $t_{ij}$ and $u_i$ arising from this condition. We firstly determine when a local solution exists:\\
If $p\mid u_i$, then by Lemma~\ref{lem: local solubility of S} we have $S_\textbf{a}(\Q_p)\neq \varnothing$.\\
If $p\mid w_{ij}$, then since $p\equiv 3\mod 4$ and $-\frac{a_i}{a_j}\in\Z_p^{\times 2}$, by Lemma~\ref{lem: local solubility of S} we have $S_\textbf{a}(\Q_p)\neq\varnothing$.\\
If $p\mid v_{ij}$, then by Lemma~\ref{lem: local solubility of S}  we have $S_\textbf{a}(\Q_p)\neq \varnothing$ if and only if either $\frac{-a_i}{a_j} \in\Z_p^{\times 4}$ or $a_k\in \Z_p^{\times 2}$. For each $p\mid v_{ij}$, fix an injection $\psi_p:\Z_p^{\times}/\Z_p^{\times 4}\xhookrightarrow{} \Z/4\Z$. Let $\boldsymbol{\zeta}=(\zeta_0,\zeta_1)\in (\Z/4\Z)^2$, then we have a unique factorisation $v_{ij} =\prod_{\boldsymbol{\zeta}\in (\Z/4\Z)^2}v^{\boldsymbol{\zeta}}_{ij}$, where
\begin{equation*}
    v^{\boldsymbol{\zeta}}_{ij}=\mathop{\prod}_{\substack{p\mid v_{ij}\\ \psi_p\left(\frac{u_i}{u_j}\right)=\zeta_0, \psi_p\left(\frac{t_{ik}}{t_{jk}}\right)=\zeta_1}} p.
\end{equation*}
Now let $\xi:=\psi_p(-\frac{A_i}{A_j})$, and let
\begin{equation*}
  \Omega_p:=\left\{(\zeta_0,\zeta_1)\in(\Z/4\Z)^2: \xi\zeta_0^2\zeta_1=0\in \Z/4\Z\right\}.
\end{equation*}
Since we have
\begin{equation*}
    -\frac{a_i}{a_j}=-\frac{A_i}{A_j}\frac{u_i^2}{u_j^2}\frac{t_{ik}}{t_{jk}},
\end{equation*}
we clearly have $\frac{-a_i}{a_j} \in\Z_p^{\times 4}$ if and only if $\boldsymbol{\zeta}\in\Omega_p$. Thus we have constructed the set $\NAM^{loc'}(T)$, which is given by
\begin{align*}
\left\{ \begin{array}{l} (\textbf{u},\textbf{v},\textbf{w})\\
\textbf{u}=(u_0,u_1,u_2)\in \NN^3\\
\textbf{v}=(v^{\boldsymbol{\zeta}}_{01},v^{\boldsymbol{\zeta}}_{02},v^{\boldsymbol{\zeta}}_{12})\in (\NN^{(\Z/4\Z)^2})^{3}\\
\textbf{w}=(w_{01},w_{02},w_{12})\in \NN^3
\end{array} : \begin{array}{l}
    \abs{A_iu_i^2\prod_{j\neq i}\prod_{\boldsymbol{\zeta}\in(\Z/4\Z)^2}v^{\boldsymbol{\zeta}}_{ij}w_{ij}}\leq T, \text{ for all }i,     \\
    \mu\left(m_{\textbf{A}}\prod_iu_i\prod_{j\neq i}\prod_{\boldsymbol{\zeta}\in(\Z/4\Z)^2}v^{\boldsymbol{\zeta}}_{ij}w_{ij}\right)^2=1,\\
    p\mid v_{ij}^{\boldsymbol{\zeta}} \implies\theta_\textbf{A} \in \Q_p^{\times2} \text{ and either} \left[a_k\in\Z_p^{\times 2}\right] \\
    \text{ or }\left[ \boldsymbol{\zeta}\in\Omega_p, \psi_p(\frac{u_i}{u_j})=\zeta_0 \text{ and } \psi_p(\frac{t_{ik}}{t_{jk}})=\zeta_1\right],\\
    p\mid w_{ij} \implies \theta_\textbf{A} \notin \Q_p^{\times2}, p\equiv 3\mod4 \text{ and } \frac{-a_i}{a_j}\in\Z_p^{\times 2},\\
     u_i^2\prod_{j\neq i}\prod_{\boldsymbol{\zeta}\in(\Z/4\Z)^2}v^{\boldsymbol{\zeta}}_{ij}w_{ij} (\mod 8m_\textbf{A}) \in M_i \text{ for all } i\\
  \end{array}\right\}.
\end{align*}
Now fix a choice of $\textbf{v}^{\boldsymbol{\zeta}},\textbf{w}$ in $\NAM^{loc'}(T)$, we will prove that for any $\textbf{u}$ in $\NAM^{loc'}(T)$ the surface $S_\textbf{a}$ having a Brauer-Main obstruction does not depend on the choice of $\textbf{u}$. To prove this, we first prove $\inv_p(\mathcal{A}_\textbf{a}(-))$ is constant at all places $p$.\par
Choose any $\textbf{u}$ in $\NAM^{loc'}(T)$, then since $S_\textbf{a}$ is everywhere locally soluble, we have $Y_\textbf{a}(\Q)\neq \varnothing$. Choose a rational point $P=[y_0:y_1:y_2:y_3]\in Y_\textbf{a}(\Q)$, then by Proposition~\ref{thm: p = 3 mod 4 and divdes two coefccients oddly implies surjective inv_p} and \ref{thm: p odd divides two coefficents to the same odd power implies inv_p surjective} we have that $\inv_p(\mathcal{A}_\textbf{a}(-))$ is constant for all $p\mid v_{ij}^{\boldsymbol{\zeta}},w_{ij}$. By Lemma~\ref{lem: p odd divdes one coff evenly} the local invariant map is also constant for all primes $p\mid u_i$. By definition of $\eta(\textbf{A},\textbf{M})$, the local invariant map is also constant for all $p\in S\cup \{p:p\mid \theta_\textbf{A}\}\cup\{\infty\}$, and hence $\inv_p(\mathcal{A}_\textbf{a}(-))$ is constant at all places $p$.\par 
Now let $\textbf{u'}\neq \textbf{u}$ be different tuple in $\NAM^{loc'}(T)$ (with $\textbf{v}^{\boldsymbol{\zeta}},\textbf{w}$ still fixed) and let $S_{\textbf{a'}}$ be the associated surface. Since $\mathcal{A}$ is an algebra on $S_{\textbf{a}}$, we have that  $\mathcal{A}_{\textbf{u'}^2\textbf{u}^{-2}}$ is an algebra on $S_{\textbf{a'}}$. If $p\nmid m_\textbf{A}\prod_iu_i\prod_{j\neq i}\prod_{\boldsymbol{\zeta}\in(\Z/4\Z)^2}v^{\boldsymbol{\zeta}}_{ij}w_{ij}$, then clearly $\mathcal{A}$ is $p$-normalised, and hence by Lemma~\ref{lem: A is p-normalised if and only if A_u is p-normalised} so is $\mathcal{A}_{\textbf{u'}^2\textbf{u}^{-2}}$. If $p\nmid m_{\textbf{A}}\prod_{i,j}\prod_{\boldsymbol{\zeta}\in(\Z/4\Z)^2}v^{\boldsymbol{\zeta}}_{ij}w_{ij}$ and  $p\mid u_k$ for some $k$, then by Lemma~\ref{lem: A is p-normalised if and only if A_u is p-normalised} $\mathcal{A}$ is $p$-normalised if and only if $\mathcal{A}_{\textbf{x}^2}$ is $p$-normalised, where $x_k=p^{-1}$ and $x_j=1$ for $j\neq k$. Clearly $\mathcal{A}_{\textbf{x}^2}$ is $p$-normalised and hence so is $\mathcal{A}$. Thus by Lemma~\ref{lem: A is p-normalised if and only if A_u is p-normalised} we have that $\mathcal{A}_{\textbf{u'}^2\textbf{u}^{-2}}$ is $p$-normalised. Therefore in either case, when $p\nmid m_{\textbf{A}}\prod_{i,j}\prod_{\boldsymbol{\zeta}\in(\Z/4\Z)^2}v^{\boldsymbol{\zeta}}_{ij}w_{ij}$, both algebras are $p$-normalised, and hence by Lemma~\ref{lem: p odd divdes one coff evenly} we have
 \begin{equation*}
     \inva{-}=\inv_p({\mathcal{A}_{\textbf{u'}^2\textbf{u}^{-2}}}(-))=0.
 \end{equation*}
If $p\mid v_{ij}^{\boldsymbol{\zeta}}$ for some $i\neq j$ and $\boldsymbol{\zeta}\in(\Z/4\Z)^2$, then $-a_0a_1a_2\in(\Q_p^{\times})^2$ and hence by the properties of the Hilbert symbol we have
\begin{equation*}
    \inva{-}=\inv_p({\mathcal{A}_{\textbf{u'}^2\textbf{u}^{-2}}}(-))=0.
\end{equation*}
If $p=\infty$, then since all entries of $\textbf{u'},\textbf{u}$ are natural numbers, there exists $b_i\in\RR$ such that $u'_iu_i^{-1}=b_i^2$ for all $i$. Thus by Lemma~\ref{lem: formula for moving between A and A_u} we have
\begin{equation*}
    \inva{-}=\inv_p({\mathcal{A}_{\textbf{u'}^2\textbf{u}^{-2}}}(-)).
\end{equation*}
If $p\mid m_{\textbf{A}}$, then $p \nmid u_i,u'_i$ for all $i$. In particular we have $u'_iu_i^{-1}\in \Z_p^{\times}$ for all $i$, and hence by Lemma~\ref{lem: formula for moving between A and A_u} we have
\begin{equation*}
    \inva{-}=\inv_p({\mathcal{A}_{\textbf{u'}^2\textbf{u}^{-2}}}(-)).
\end{equation*}
If $p\mid w_{ij}$ for some $i\neq j$, then by Lemma~\ref{lem: formula for when p divdies w_ij} we have
\begin{equation*}
    \inv_p(\mathcal{A}_{\textbf{u}'^2\textbf{u}^{-2}}(-))=\inva{-} +\frac{\left(\frac{u_i'u_i^{-1}u_{j}'u_j^{-1}}{p}\right)-1}{4}.
\end{equation*}
Therefore summing over all places $p$ we have
\begin{equation*}
    \sum_p \inv_p(\mathcal{A}_{\textbf{u}'^2\textbf{u}^{-2}}(-)) = \sum_p  \inva{-} +\frac{\prod_{i<j}\left(\frac{u_i'u_i^{-1}u_{j}'u_j^{-1}}{w_{ij}}\right)-1}{4}.
\end{equation*}
Thus the exists some function that does not depend on $\textbf{u}$, say $\delta:N_{\textbf{A,M}}^{loc'}(T)\rightarrow \{-1,1\}$, such that $S_\textbf{a}$ has a Brauer-Manin obstruction if and only if
\begin{equation*}
    \prod_{i<j}\left(\frac{u_iu_{j}}{w_{ij}}\right)=\delta(\textbf{A},\textbf{M},\textbf{v}^{\boldsymbol{\zeta}},\textbf{w}).
\end{equation*}
Therefore the indicator function on $\NAM^{loc'}(T)$ for whether $S_\textbf{a}$ has a Brauer-Manin obstruction induced by $\mathcal{A}$ is given by
\begin{equation*}
    \frac{1}{2}+\frac{\delta(\textbf{A},\textbf{M},\textbf{v}^{\boldsymbol{\zeta}},\textbf{w})}{2}\prod_{i<j}\left(\frac{u_iu_{j}}{w_{ij}}\right).
\end{equation*}
Thus to prove the result it suffices to prove
\begin{equation*}
    \mathop{\sum\sum\sum}_{\substack{({\textbf{v}^{\boldsymbol{\zeta}},\textbf{w},\textbf{u}})\in \NAM^{loc'}(T)}} \delta(\textbf{A},\textbf{M},\textbf{v}^{\boldsymbol{\zeta}},\textbf{w}) \prod_{i<j} \left(\frac{u_iu_j}{w_{ij}}\right)=O_{\epsilon}\left(\frac{T^\frac{3}{2} (\log T)^{\frac{5}{2}M_\textbf{A}}}{\abs{\theta_\textbf{A}}^{\frac{1}{2}-\epsilon}}\right).
\end{equation*}
Similar to the proof of \cite[Lem.~5.6]{TimSantens}, by applying \cite[Thm.~4.22]{TimSantens} with $q_{osc}=O(1)$ and $A=6$ we have that this sum is bounded by a sum which counts the elements in $\NAM^{loc'}(T)$ which are contained in the union of the two sets defined by the two conditions below:
\begin{enumerate}
    \item \label{case: case <log theta} $u_k\leq (\log \abs{\theta_\textbf{A}}^{-\frac{1}{2}}T^\frac{3}{2})^C$ for at least $2$ different $k$;
    \item \label{case: case >log theta} $w_{ij} \leq (\log \abs{\theta_\textbf{A}}^{-\frac{1}{2}}T^\frac{3}{2})^C$ for all $i\neq j$,
\end{enumerate}
where $C>0$ is some constant. In particular the sum over the set defined by condition~\eqref{case: case <log theta} is bounded by \eqref{eq: first downward open set} in Lemma~\ref{lem: bound for downward open sets for reduction} and the sum over the set defined by condition~\eqref{case: case >log theta} is bounded by \eqref{eq: second downward open set} in Lemma~\ref{lem: bound for downward open sets for reduction}. Therefore the result holds.
\end{proof}
\subsection{Proof of the main theorem}
\label{subsec: proof of the main theorem}
Before proving Theorem~\ref{thm: main theorem split into squares and non squares} we state the following lemma, which we will prove in \S\ref{subsec: The main term}, as well as define a function:
\begin{lemma}
\label{lem: assymptotic formula for local set}
For all $(\textbf{A},\textbf{M})\in$ \textbf{$\Phi(T)$}$\times$$\Psi_\textbf{A}$ there exists a positive constant $0<Q_{\textbf{A}}\ll_{\epsilon} \abs{\theta_{\textbf{A}}}^{\epsilon}$ for all $\epsilon>0$ such that for $T>2$ we have
    \begin{equation*}
    \#\NAM^{loc}(T)=\left(Q_{\textbf{A}}+O_{\epsilon}\left(\frac{\abs{\theta_{\textbf{A}}}^{\epsilon}}{\left(\log T\right)^{\frac{6}{5}M_\textbf{A}}}\right)\right)\frac{T^2\left(\log T\right)^{3M_\textbf{A}}}{\abs{\theta_{\textbf{A}}}^{\frac{1}{2}}}.
    \end{equation*}
\end{lemma}
\begin{definition}
Let $\tau_n(m)$ denote the multiplicative function defined on primes by
    \begin{equation*}
        \tau_n(p^k)= \begin{pmatrix}
            n+k-1 \\
             k 
        \end{pmatrix}.
    \end{equation*}
We will use the divisor bound in the proof of Theorem~\ref{thm: main theorem split into squares and non squares}, namely that for all $\epsilon>0$ we have $\tau_n(m)\ll_{n,\epsilon} m^{\epsilon}$.
\end{definition}

\begin{proof}[Proof of Theorem~\ref{thm: main theorem split into squares and non squares}]
We proceed as in the proof of \cite[Thm.~1.2]{TimSantens}. Firstly by Lemma~\ref{lem: reducing N^Br to sum of N^Br_A,M} and Lemma~\ref{lem: theta_a a square bounded} we have
\begin{equation*}
 \#N^{\Br}_{=\Box}(T)=\mathop{\sum}_{\substack{(\textbf{A},\textbf{M})\in \textbf{$\Phi(T)$}\times\Psi_\textbf{A}\\ \theta_\textbf{A} \in \Q^{\times2}}}  O\left(\frac{T^\frac{3}{2}\log T}{\abs{\theta_\textbf{A}}^{\frac{1}{2}}}\right).   
\end{equation*}
By Lemma~\ref{lem: reducing N^Br to sum of N^Br_A,M},~\ref{lem: reducing N^Br to N^A},~\ref{lem: reducing N^A to N^loc} and \ref{lem: assymptotic formula for local set} we have
\begin{equation*}
  \#N^{\Br}_{=-\Box}(T)=T^\frac{3}{2} (\log T)^\frac{3}{2}\mathop{\sum}_{\substack{(\textbf{A},\textbf{M})\in \textbf{$\Phi(T)$}\times\Psi_\textbf{A}\\ \theta_\textbf{A} \in -\Q^{\times 2}}} \frac{\eta(\textbf{A}, \textbf{M})}{2\abs{\theta_\textbf{A}}^{\frac{1}{2}}}\left(Q_\textbf{A}+O_{\epsilon}\left(\frac{\abs{\theta_\textbf{A}}^{\epsilon}}{(\log T)^\frac{3}{5}}\right)\right),      
\end{equation*}
\begin{equation*}
   \#N^{\Br}_{\neq\pm\Box}(T)=T^\frac{3}{2} (\log T)^\frac{9}{8}\mathop{\sum}_{\substack{(\textbf{A},\textbf{M})\in \textbf{$\Phi(T)$}\times\Psi_\textbf{A}\\ \theta_\textbf{A} \not\in \pm\Q^{\times2}}}  \frac{\eta(\textbf{A}, \textbf{M})}{2\abs{\theta_\textbf{A}}^{\frac{1}{2}}}\left(Q_\textbf{A}+O_{\epsilon}\left(\frac{\abs{\theta_\textbf{A}}^{\epsilon}}{(\log T)^\frac{9}{20}}\right)\right).     
\end{equation*}
Thus to prove this result it suffices to show the sums
\begin{equation*}
\mathop{\sum}_{\substack{(\textbf{A},\textbf{M})\in \textbf{$\Phi(T)$}\times\Psi_\textbf{A}\\ \theta_\textbf{A} \in \Q^{\times2}}}   \abs{\theta_\textbf{A}}^{-\frac{1}{2}},  
\end{equation*}
\begin{equation}
\label{eq: first sum for main theorem}
  \mathop{\sum}_{\substack{(\textbf{A},\textbf{M})\in \textbf{$\Phi(T)$}\times\Psi_\textbf{A}\\ \theta_\textbf{A} \in -\Q^{\times 2}}} \frac{\eta(\textbf{A}, \textbf{M})Q_\textbf{A}}{2\abs{\theta_\textbf{A}}^{\frac{1}{2}}},
\end{equation}
\begin{equation}
\label{eq: second sum for main theorem}
 \mathop{\sum}_{\substack{(\textbf{A},\textbf{M})\in \textbf{$\Phi(T)$}\times\Psi_\textbf{A}\\ \theta_\textbf{A} \not\in \pm\Q^{\times2}}}  \frac{\eta(\textbf{A}, \textbf{M})Q_\textbf{A}}{2\abs{\theta_\textbf{A}}^{\frac{1}{2}}},  
\end{equation}
converge and that the latter two converge to a positive constant.\par
Firstly assume the sums converge, then to prove the sums~\eqref{eq: first sum for main theorem} and \eqref{eq: second sum for main theorem} converge to a positive constant it suffices to prove $\eta(\textbf{A},\textbf{M})\neq 0$ for some $(\textbf{A},\textbf{M})\in$ \textbf{$\Phi(T)$}$\times\Psi_\textbf{A}$. In particular in both cases it suffices to find a surface which is everywhere locally soluble and $\mathcal{A}$ induces a Bauer-Manin obstruction. The examples constructed in \cite[Ex.~6]{On_the_arithmetic_of_del_Pezzo_surfaces_of_degree_2} and \cite[Ex.~4]{On_the_arithmetic_of_del_Pezzo_surfaces_of_degree_2} give such surfaces for $\theta_\textbf{A}\not\in\pm\Q^{\times 2}$ and $\theta_\textbf{A}\in-\Q^{\times 2}$, respectively.\par
Now since $0<Q_{\textbf{A}}\ll_{\epsilon} \abs{\theta_{\textbf{A}}}^{\epsilon}$, by taking an upper bound of each sum, in order to prove each sum converges it suffices to prove the sum
\begin{equation*}
    \mathop{\sum}_{\substack{(\textbf{A},\textbf{M})\in \textbf{$\Phi(T)$}\times\Psi_\textbf{A}\\
    \abs{\theta_\textbf{A}}\geq T}}\abs{\theta_\textbf{A}}^{-\frac{1}{2}+\epsilon},
\end{equation*}
tends to zero as $T\rightarrow \infty$ for sufficiently small $\epsilon>0$. For each $\textbf{A}\in \Phi(T)$ the number of $\textbf{M}\in\Psi_\textbf{A}$ can be bounded by the number of tuples $(n_0,n_1,n_2)$ of cosets of $(\Z/8m_\textbf{A}\Z)^{\times 4}$ in $(\Z/8m_\textbf{A}\Z)^{\times}$. The number of such cosets is $\ll\tau_4(\abs{\theta_\textbf{A}})$, and hence there are $\ll\tau_4(\abs{\theta_\textbf{A}})^3\ll_{\epsilon}\abs{\theta_\textbf{A}}^{\epsilon}$ possible $\textbf{M}\in \Psi_\textbf{A}$. Therefore we have
\begin{equation*}
    \mathop{\sum}_{\substack{(\textbf{A},\textbf{M})\in \textbf{$\Phi(T)$}\times\Psi_\textbf{A}\\
    \abs{\theta_\textbf{A}}\geq T}}\abs{\theta_\textbf{A}}^{-\frac{1}{2}+\epsilon}\ll_{\epsilon}\mathop{\sum}_{\substack{\textbf{A}\in \textbf{$\Phi(T)$}\\
    \abs{\theta_\textbf{A}}\geq T}}\abs{\theta_\textbf{A}}^{-\frac{1}{2}+2\epsilon}.
\end{equation*}
Now for any $\textbf{A}\in \Phi(T)$ we can write $\theta_\textbf{A}=gh_3^3h_4^4h_5^5$, where if $p\mid g$, then $p\in S$ and $\nu_p(g)\leq 2$ (and hence there are finitely many choices for $g$). For fixed $h_3,h_4,h_5\in\Z\setminus\{0\}$, there are $\tau_3(gh_3^3h_4^4h_5^5)\ll_{\epsilon}\abs{gh_3^3h_4^4h_5^5}^{\epsilon}$ possible tuples $\textbf{A}\in \Phi(T)$ such that $\theta_\textbf{A}=gh_3^3h_4^4h_5^5$. In particular we have
\begin{align*}
   \mathop{\sum}_{\substack{\textbf{A}\in \textbf{$\Phi(T)$}\\
    \abs{\theta_\textbf{A}}\geq T}}\abs{\theta_\textbf{A}}^{-\frac{1}{2}+2\epsilon} \ll \mathop{\sum}_{\substack{g,h_3,h_4,h_5\\ \abs{gh_3^3h_4^4h_5^5}>T}} \abs{gh_3^3h_4^4h_5^5}^{-\frac{1}{2}+3\epsilon}&\ll_{\epsilon}\mathop{\sum}_{\substack{g,h_3,h_4,h_5\\ \abs{gh_3^3h_4^4h_5^5}>T}} (h_3)^{-\frac{1}{2}+9\epsilon}(h_4^4h_5^5)^{-\frac{1}{2}+3\epsilon}\\
    &\ll \mathop{\sum}_{\substack{h_4,h_5}} \left(\frac{h_4^{\frac{4}{3}}h_5^{\frac{5}{3}}}{T^\frac{1}{3}}\right)^{\frac{1}{2}-9\epsilon}\left(h_4^4h_5^5\right)^{-\frac{1}{2}+3\epsilon}\\
    &\ll T^{-\frac{1}{6}+3\epsilon}\mathop{\sum}_{\substack{h_4,h_5}}h_4^{-2+\frac{2}{3}+12\epsilon-12\epsilon}h_5^{-\frac{5}{2}+\frac{5}{6}+15\epsilon-15\epsilon}\\
    &\ll T^{-\frac{1}{6}+3\epsilon}.
\end{align*}
Therefore taking $\epsilon<\frac{1}{18}$ we see that this sum tends to zero as $T\rightarrow \infty$, and hence the result holds.
\end{proof}
\subsection{The error terms}
\label{subsec: The error terms}
In this section we will convert $\#\NAM^{loc}(T)$ into a sum and simplify it. To do this, we will construct functions that will give us the conditions on $\NAM^{loc}(T)$. We split this section into subsections which will deal with each condition individually.

\subsubsection{Converting $\#\NAM^{loc}(T)$ into a sum}
\begin{definition}
Let $m\in\NN$. We say a subset $U\subseteq \NN^m$ is \textbf{downward-closed} if $(x_n)_{n=1}^{m}\in U$ implies $\{(y_n)_{n=1}^m\in \NN^m: y_n\leq x_n, \text{ for all } 1\leq n\leq m\}\subseteq U$.     
\end{definition}
Let $U$ be a downward-closed subset of $\N_{\textbf{A},\textbf{M}}(T)$, and let $\NAM^{loc,U}(T):=\NAM^{loc}(T)\cap U$, then we have
\begin{equation}
\label{eq: N_A,M^U(T) is a sum first eq}
    \#\NAM^{loc,U}(T)=\mathop{\sum\sum\sum}_{\substack{(\textbf{u},\textbf{v},\textbf{w}})\in \NAM^{loc,U}(T)}1.
\end{equation}

\subsubsection{Dealing with \textbf{M} condition on $\NAM^{loc}(T)$}
Let 
\begin{equation*}
    \Gamma_\textbf{A}:=(\Z/8m_\textbf{A})^{\times}/(\Z/8m_\textbf{A})^{\times 4},
\end{equation*}
then we can write the condition induced by \textbf{M} on $\NAM^{loc,U}(T)$ as the sum of characters
\begin{equation*}
    \frac{1}{\abs{\Gamma_\textbf{A}}^3}\prod_i\sum_{\chi_i\in \Gamma_\textbf{A}^{\vee}}\overline{\chi_i}(M_i)\chi_i\left(u_i^2\prod_{j\neq i} v_{ij}w_{ij}\right).
\end{equation*}
Substituting this into \eqref{eq: N_A,M^U(T) is a sum first eq} we have
\begin{equation}
\label{eq: N_A,M^loc(T) second simplified form (M condition removed)}
     \#\NAM^{loc,U}(T)=\frac{1}{\abs{\Gamma_\textbf{A}}^3}\sum_{\boldsymbol{\chi}\in (\Gamma_\textbf{A}^{\vee})^3}S_\textbf{A}^U(\boldsymbol{\chi}, T)\prod_i\overline{\chi_i}(M_i),
\end{equation}
where
\begin{equation}
\label{eq: S_A(X,T) first simplified form (M condition removed)}
    S_\textbf{A}^U(\boldsymbol{\chi}, T):=\mathop{\sum\sum\sum}_{\substack{(\textbf{u},\textbf{v},\textbf{w}})\in \mathop{\bigsqcup}_{\substack{N\in\Omega}}N_{\textbf{A,\textbf{N}}}^{loc,U}(T)}\prod_i\chi_i\left(u_i^2\prod_{j\neq i} v_{ij}w_{ij}\right).
\end{equation}
We will let $S_\textbf{A}(\boldsymbol{\chi},T):=S_\textbf{A}^{\NAM^{loc}(T)}(\boldsymbol{\chi}, T)$.
\begin{remark}
\label{rem: extending chi_i to all of Z/8m_AZ}
For $\chi_i\in\Gamma_\textbf{A}^{\vee}$ we can extend $\chi_i$ to the map on $(\Z/8m_\textbf{A})/(\Z/8m_\textbf{A})^{4}$ defined by
\begin{equation*}
    \tilde{\chi_i}(x)=\begin{cases}
        \chi_i(x), &\text{ if } x\in \Gamma_\textbf{A},\\
        0, &\text{ otherwise}.
    \end{cases}
\end{equation*}
By definition, if $u_i\prod_{j\neq i} v_{ij}w_{ij}$ is not coprime to $m_\textbf{A}$ for some $i$, then
\begin{equation*}
  \Tilde{\chi_\ell}\left(u_i^2\prod_{j\neq i} v_{ij}w_{ij}\right)=0,
\end{equation*}
for all $\ell$. In particular, the condition that $u_i\prod_{j\neq i} v_{ij}w_{ij}$ is coprime to $m_\textbf{A}$ for all $i$ is contained in this condition induced by $\textbf{M}$. By abuse of notation we will denote $\Tilde{\chi_i}$ by $\chi_i$. 
\end{remark}
\subsubsection{Dealing with local solubility condition on $\NAM^{loc}(T)$}
We now deal with the existence of $\Q_p$-points when $p\nmid m_\textbf{A}$. We first determine criteria for when a $\Q_p$-point exists:
\begin{lemma}
\label{lem: conditions for local solublity in counting}
The condition $S_\textbf{a}(\Q_p)\neq \varnothing$ for $p\nmid m_\textbf{A}$ in $\NAM^{loc}(T)$ is equivalent to the conditions
\begin{enumerate}
    \item \label{con: 1} $p\mid w_{ij} \implies \left[p\equiv 3\mod 4, \frac{-a_i}{a_j}\in\Q_p^{\times 2} \text{ and } \theta_\textbf{A}\not\in\Q_p^{\times 2} \right]$;
    \item \label{con: 2} $p\mid v_{ij} \implies \left[ a_k\in \Z_p^{\times 2}\right]$.
\end{enumerate}
\end{lemma}
\begin{proof}
We split into cases depending on which variable a prime $p\nmid m_\textbf{A}$ divides:\\
If $p\nmid m_\textbf{A}\prod_iu_i\prod_{i\neq j}v_{ij}w_{ij}$, then by Lemma~\ref{lem: local solubility of S} we have that $S_\textbf{a}(\Q_p)\neq \varnothing$.\\
If $p\mid u_k$, then by Lemma~\ref{lem: local solubility of S} we have that $S_\textbf{a}(\Q_p)\neq \varnothing$.\\
If $p\mid w_{ij}$, $p\equiv 3\mod 4$ and $\frac{-a_i}{a_j}\in \Z_p^{\times 2}$, then either $\pm \sqrt{\frac{-a_i}{a_j}} \in \Z_p^{\times 2}$, and hence by Lemma~\ref{lem: local solubility of S} we have that $S_\textbf{a}(\Q_p)\neq\varnothing$. In particular the condition $S_\textbf{a}(\Q_p)\neq\varnothing$ when $p\mid w_{ij}$ is contained in the condition $p\mid w_{ij} \implies \left[p\equiv 3\mod 4, \frac{-a_i}{a_j}\in\Q_p^{\times 2} \text{ and } \theta_\textbf{A}\not\in\Q_p^{\times 2} \right]$.\\
If $p\mid v_{ij}$, then since $\theta_\textbf{A}\in\Q_p^{\times 2}$ we have $\frac{-a_i}{a_j}=a_k \in \Z_p^{\times}/\Z_p^{\times 2}$. Thus by Lemma~\ref{lem: local solubility of S} for a $\Q_p$-point to exist, we require the additional condition that $p\mid v_{ij} \implies \left[ a_k\in \Z_p^{\times 2}\right]$.
\end{proof}
We now construct functions to encapsulate conditions \eqref{con: 1} and \eqref{con: 2} in Lemma~\ref{lem: conditions for local solublity in counting}:
\begin{definition}
Let $\gamma(w;x,y)$ be the indicator function for the property
\begin{equation*}
    w \text{ square-free, } p\mid w \implies p\equiv3 \mod 4 \text{ and } x,y\not\in\Q_p^{\times 2},
\end{equation*}
and let $\omega(v;x,y)$ be the indicator function for the property
\begin{equation*}
    v \text{ square-free, } p\mid v \implies x,y\in\Q_p^{\times 2}.
\end{equation*}    
\end{definition}
By Lemma~\ref{lem: conditions for local solublity in counting} we know that $\omega(v_{ij};\theta_\textbf{A},a_k)$ and $\gamma(w_{ij};\theta_\textbf{A},a_k)$ encode the conditions for local solubility arising from $\NAM^{loc}(T)$. Before substituting these functions into \eqref{eq: S_A(X,T) first simplified form (M condition removed)} we first prove they can both be written in terms of simpler $S$-frobenian multiplicative functions (\cite[Def.~2.7]{LoughranDaniel2023Fmfa}), the definition of which we recall below:
\begin{definition}
\label{def: frob mult functions}
Let $S$ be a finite set of primes of $\QQ$. We say a multiplicative function $f: \NN \rightarrow\CC$ is a \textbf{S-frobenian multiplicative function} if it satisfies the following:
\begin{itemize}
    \item There exists a constant $N\in\NN$ such that $\abs{f(p^k)}\leq N^k$, for all primes $p$ and $k\in\NN$;
    \item For all $\epsilon>0$ there exists a constant $C_\epsilon\in\NN$ such that $\abs{f(n)}\leq C_\epsilon n^\epsilon$, for all $n\in \NN$;
    \item There exists a Galois extension $L/\Q$ and a class function $\phi:\Gal(L/\QQ)\rightarrow\CC$ such that every prime that ramifies is contained in $S$ and for every prime $p\notin S$ we have
    \begin{equation*}
        f(p)=\phi(\operatorname{Frob}_p).
    \end{equation*}
\end{itemize}
Given such a S-frobenian multiplicative function $f$, we define its \textbf{mean} as
\begin{equation*}
    \frac{1}{\abs{\Gal(L/\QQ)}}\sum_{\gamma\in\Gal(L/\QQ)}\phi(\gamma).
\end{equation*}
\end{definition}
\begin{definition}
Let $\alpha(w;x)$ and $\beta(v;x)$ be the multiplicative functions that take the following values on primes:
 \begin{equation*}
    \alpha(p;x)=\begin{cases}
        \frac{1}{2}, \text{ if } p\equiv 3\mod 4 \text{ and } x\not\in\Q_p^{\times 2},\\
        0, \text{ otherwise,}
    \end{cases}
\end{equation*}
and
\begin{equation*}
    \beta(p;x)=\begin{cases}
        \frac{1}{2}, \text{ if } x\in\Q_p^{\times 2},\\
        0, \text{ otherwise.}
    \end{cases}
\end{equation*}   
\end{definition}
\begin{lemma}
Let $x\in \Z\setminus\{0\}$ such that $x$ is not a square, then $\alpha(-;x)$ is a $\{p:p\mid 2x\}-$frobenian multiplicative function of mean $\frac{1}{8}$ when $x\not\in\pm\Q^{\times 2}$ and mean $\frac{1}{4}$ when $-x\in\Q^{\times 2}$.    
\end{lemma}
\begin{proof}
Clearly $\alpha(-;x)$ is multiplicative and satisfies the first two properties given in Definition~\ref{def: frob mult functions}. Let $L_1:=\Q(\sqrt{-1})$, $L_2:=\Q(\sqrt{x})$, and let $L:=L_1L_2$. Furthermore, let $<\sigma>:=\Gal(L_1/\Q)$ and let $<\tau>:=\Gal(L_2/\Q)$. Clearly all primes that ramify in $L$ are contained in $S:=\{p:p\mid 2x\}$. For $p\notin S$ prime, we have that $p\equiv3\mod 4$ and $x\notin \Q_p^{\times 2}$ if and only if $p$ is inert in $L_1$ and $L_2$. Furthermore for any such $p\notin S$, we have that $\operatorname{Frob}_p=(\sigma,\tau)$, when $x\not\in\pm\Q^{\times 2}$, and $\operatorname{Frob}_p=\sigma$, when $-x\in\Q^{\times 2}$. Therefore in the case $x\not\in\pm\Q^{\times 2}$, define the class function $\phi:\Gal(L/\Q)\rightarrow\CC$ by $\phi((\sigma,\tau))=\frac{1}{2}$, and $\phi(\delta)=0$, for all $\delta\neq (\sigma,\tau)\in\Gal(L/\Q)$. In the case $-x\in\Q^{\times 2}$, define the class function $\phi:\Gal(L/\Q)\rightarrow\CC$ by $\phi(\sigma)=\frac{1}{2}$, and $\phi(1)=0$. By definition we have that the mean of $\alpha(-;x)$ is $\frac{1}{2\abs{\Gal(L/\Q)}}$, and hence the result follows.  
\end{proof}

\begin{lemma}
 Let $x\in \Z\setminus\{0\}$ such that $x$ is not a square, then $\beta(-;x)$ is a $\{p:p\mid 2x\}-$frobenian multiplicative function of mean $\frac{1}{4}$ when $x\not\in\pm\Q^{\times 2}$ or when $x\in-\Q^{\times 2}$.       
\end{lemma}
\begin{proof}
Clearly $\beta(-;x)$ is multiplicative and satisfies the first two properties given in Definition~\ref{def: frob mult functions}. Let $L:=\Q(\sqrt{x})$. Clearly all primes that ramify in $L$ are contained in $S:=\{p:p\mid 2x\}$. For $p\notin S$ prime, we have that $x\in \Q_p^{\times 2}$ if and only if $p$ splits in $L$. Furthermore, for any split $p\notin S$ we have that $\operatorname{Frob}_p=1\in\Gal(L/\Q)$. Thus define the class function $\phi:\Gal(L/\Q)\rightarrow\CC$ by $\phi(1)=\frac{1}{2}$, and $\phi(\sigma)=0$, for all $\sigma\neq 1\in\Gal(L/\Q)$. By definition we have that the mean of $\beta(-;x)$ is $\frac{1}{2\abs{\Gal(L/\Q)}}$, and hence the result follows.  
\end{proof}

\begin{lemma}
\label{lem: lambda formula}
For $x,y,w\in\ZZ\setminus\{0\}$, we have
\begin{equation*}
\gamma(w;x,y)=\mu(w)^2\left[\sum_{hf=w}\alpha(h;y)\alpha(f;y)\left(\frac{-x}{f}\right)\right].
\end{equation*}    
\end{lemma}
\begin{proof}
Firstly both sides are multiplicative in $w$ and are both $0$ if $w$ is not square-free. Thus it suffices to prove this equality holds when $w$ is a prime, say $p$.\par
For a given prime $p$, if $p\not \equiv 3\mod 4$, then both sides are $0$, so we may assume $p\equiv 3\mod 4$. Furthermore, we can assume $y\not\in\Q_p^{\times 2}$ as otherwise both sides are $0$. In this case the right side is given by
\begin{equation*}
    \alpha(p;y)+\alpha(p;y)\left(\frac{-x}{p}\right)=\frac{1}{2}+\frac{1}{2}\left(\frac{-x}{p}\right)=\frac{1}{2}\left[1-\left(\frac{x}{p}\right)\right].
\end{equation*}
Thus this equation is $1$ if and only if $x\not\in\Q_p^{\times 2}$ (and $0$ otherwise), and hence equality holds.
\end{proof}

\begin{lemma}
\label{lem: omega formula}
For $x,y,v\in\ZZ\setminus\{0\}$, we have
\begin{equation*}
\omega(v;x,y)=\mu(v)^2\left[\sum_{de=v}\beta(d;y)\beta(e:y)\left(\frac{x}{e}\right)\right].
\end{equation*}    
\end{lemma}
\begin{proof}
Firstly both sides are multiplicative in $v$ and are both $0$ if $v$ is not square-free. Thus it suffices to prove this equality holds when $v$ is a prime, say $p$. Furthermore, we can assume $y\in\Q_p^{\times 2}$ as otherwise both sides are $0$. In this case the right side is given by
\begin{equation*}
    \beta(p;y)+\beta(p;y)\left(\frac{x}{p}\right)=\frac{1}{2}+\frac{1}{2}\left(\frac{x}{p}\right).
\end{equation*}
Thus this equation is $1$ if and only if $x\in\Q_p^{\times 2}$ (and $0$ otherwise), and hence equality holds.
\end{proof}
Before substituting the functions in Lemma~\ref{lem: lambda formula} and \ref{lem: omega formula} into \eqref{eq: S_A(X,T) first simplified form (M condition removed)} we provide the following definition:
\begin{definition}
Let $q_{\text{frob}}\in \ZZ$. Let $(u_n)_{n\in I}$ be some tuple of variables indexed by some finite set $I$. Let $u_{\widehat{n}}:=(u_\ell)_{\ell\in I, \ell\neq n}$ and let $u_{\widehat{nm}}:=(u_\ell)_{\ell\in I, \ell\neq n,m}$ . Let $f:\NN^I\rightarrow \CC$ be some function. Then
\begin{enumerate}
    \item For $n\neq m\in I$, we say $u_n$ and $u_m$ are \textbf{linked} if there exists two functions $a(-;u_{\widehat{nm}}),b(-;u_{\widehat{nm}}):\NN\rightarrow\CC$ which satisfy the following:
    \begin{itemize}
        \item $f\left((u_n)_{n\in I}\right)=a(u_n;u_{\widehat{nm}})b(u_m;u_{\widehat{nm}})\left(\frac{u_n}{u_m}\right)\left(\frac{u_m}{u_n}\right)$;
        \item $\abs{a(u_n;u_{\widehat{nm}})}\leq 1$;
        \item $\abs{b(u_m;u_{\widehat{nm}})}\leq 1$.
    \end{itemize}
\item For $n\in I$ we say $u_n$ is \textbf{frobenian} if there exists a constant $C(u_{\widehat{n}})$ which only depends on $u_{\widehat{n}}$ and a $\{p;p\mid q_{\text{frob}}\}\cup\{p:p\mid u_m, m\neq n\}$-frobenian multiplicative function $\rho(-,u_{\widehat{n}})$ which satisfy the following:
    \begin{itemize}
        \item $f\left((u_n)_{n\in I}\right)=C(u_{\widehat{n}})\rho(u_n,u_{\widehat{n}})$;
        \item $\abs{C(u_{\widehat{n}})}\leq 1$;
        \item The conductor of $\rho(-,u_{\widehat{n}})$ is at most $q_{\text{frob}}\prod_{u_m \text{linked to }u_n}u_m^6$;
        \item If there exists $u_m$ linked to $u_n$ such that $u_m\neq 1$, then the mean of $\rho(-,u_{\widehat{n}})$ is $0$.
    \end{itemize}
\item For $\epsilon>0$ we say $(u_n)_{n\in I}$ satisfies $\star_{\epsilon}$ if the following holds:
\begin{itemize}
    \item For every frobenian variable $u_n$, if $u_n>e^{\left(\frac{3}{2}\log T\right)^\epsilon}$, then all variables linked to $u_n$ are equal to $1$.
\end{itemize}
\end{enumerate}  
\end{definition}
\begin{remark}
By \cite[Lem.~4.20]{TimSantens} the function $\psi(x,y)=\left(\frac{x}{y}\right)\left(\frac{y}{x}\right)$ is a $(6,2^{24})$-oscillating bilinear characters (see \cite[Def.~4.19]{TimSantens} for the definition), and hence our definition of linked variables is compatible with \cite[Thm.~4.22]{TimSantens} by taking $A=6$ and $q_{\text{osc}}=O(1)$. 
\end{remark}
\begin{lemma}
\label{lem: formula for downward-closed subset N^(loc,U)(T)}
For $T>2$ we have
\begin{align*}
     S^U_\textbf{A}(\boldsymbol{\chi},T)=&\mathop{\sum\sum\sum}_{\substack{(\textbf{u,h,d})\in U\\ \mu(\prod_iu_i\prod_{j\neq i}h_{ij}d_{ij})^2=1}}\prod_i \chi_i(u_i^2)\prod_{j\neq i}\chi_i(\chi_{j}(h_{ij}d_{ij}))\alpha(h_{ij};\theta_\textbf{A})\beta(d_{ij};\theta_\textbf{A})\\
     &+O_{\epsilon}\left(\frac{T^\frac{3}{2}(\log T)^{\frac{5}{2}M_\textbf{A}}}{\abs{\theta_\textbf{A}}^{\frac{1}{2}-\epsilon}}\right).
\end{align*}
\end{lemma}
\raggedbottom
\begin{proof}
We adapt the proof of \cite[Lem.~5.9]{TimSantens}. Write $\textbf{v}=\textbf{d}\textbf{e}$ and $\textbf{w}=\textbf{h}\textbf{f}$. Let $d_{ij}$ and $e_{ij}$ denote the parts of $\textbf{d}$ and $\textbf{e}$ such that $d_{ij}e_{ij}=v_{ij}$, and similarly let $h_{ij}$ and $f_{ij}$ denote the parts of $\textbf{h}$ and $\textbf{f}$ such that $h_{ij}f_{ij}=w_{ij}$. Substituting Lemma~\ref{lem: lambda formula} and Lemma~\ref{lem: omega formula} into \eqref{eq: S_A(X,T) first simplified form (M condition removed)} we have
\begin{align*}
   S^U_\textbf{A}(\boldsymbol{\chi},T)&=\mathop{\sum\sum\sum\sum\sum}_{\substack{(\textbf{u,hf,de})\in U\\ \mu(\prod_iu_i\prod_{j\neq i}h_{ij}f_{ij}d_{ij}e_{ij})^2=1}}\prod_i \chi_i(u_i^2)\prod_{j\neq i}\Big[\chi_i(\chi_{j}(h_{ij}f_{ij}d_{ij}e_{ij}))\\
   &\times \alpha(h_{ij};\theta_\textbf{A})\alpha(f_{ij};\theta_\textbf{A})\left(\frac{-A_k\prod_{\ell\neq k}d_{\ell k}e_{\ell k}f_{\ell k}h_{\ell k}}{f_{ij}}\right)\\
   &\times\beta(d_{ij};\theta_\textbf{A})\beta(e_{ij};\theta_\textbf{A})\left(\frac{A_k\prod_{\ell\neq k}d_{\ell k}e_{\ell k}f_{\ell k}h_{\ell k}}{e_{ij}}\right)\Big].
\end{align*}
We will let $F$ denote the function within the summation. Clearly we have that each of the pairs of variables $(f_{ki},f_{kj}), (f_{ki},d_{kj}),(f_{ki},h_{kj}),(e_{ki},e_{kj}), (e_{ki},d_{kj})$ and $(e_{ki},h_{kj})$ are linked. We now prove every variable is frobenian (with $q_{frob}=O(1)\abs{\theta_\textbf{A}}^{O(1)}$).\par
If we fix all variables other than $u_i$, then $F$ can be written as a constant whose absolute value is bounded by $1$ (which does not depend on $u_i$) multiplied by $\chi_i(u_i^2)$.\par
If we fix all variables other than $h_{ij}$, then $F$ can be written as a constant whose absolute value is bounded by $1$ (which does not depend on $h_{ij}$) multiplied by the function
\begin{equation*}
    \chi_i(\chi_{j}(h_{ij}))\alpha(h_{ij};\theta_\textbf{A})\left(\frac{h_{ij}}{\prod_{\ell\neq j}f_{i\ell}e_{i\ell}}\right).
\end{equation*}
If we fix all variables other than $d_{ij}$, then $F$ can be written as a constant whose absolute value is bounded by $1$ (which does not depend on $d_{ij}$) multiplied by the function
\begin{equation*}
    \chi_i(\chi_{j}(d_{ij}))\beta(d_{ij};\theta_\textbf{A})\left(\frac{d_{ij}}{\prod_{\ell\neq k}f_{\ell k}e_{\ell k}}\right).
\end{equation*}
If we fix all variables other than $f_{ij}$, then $F$ can be written as a constant whose absolute value is bounded by $1$ (which does not depend on $f_{ij}$) multiplied by the function
\begin{equation*}
    \chi_i(\chi_{j}(f_{ij}))\alpha(f_{ij};\theta_\textbf{A})\left(\frac{f_{ij}}{\prod_{\ell\neq k}f_{\ell k}e_{\ell k}}\right)\left(\frac{-A_k\prod_{\ell\neq k}d_{k\ell}e_{k\ell}f_{k\ell}h_{k\ell}}{f_{ij}}\right).
\end{equation*}
If we fix all variables other than $e_{ij}$, then $F$ can be written as a constant whose absolute value is bounded by $1$ (which does not depend on $e_{ij}$) multiplied by the function
\begin{equation*}
    \chi_i(\chi_{j}(e_{ij}))\beta(e_{ij};\theta_\textbf{A})\left(\frac{e_{ij}}{\prod_{\ell\neq k}f_{\ell k}e_{\ell k}}\right)\left(\frac{A_k\prod_{\ell\neq k}d_{k\ell}e_{k\ell}f_{k\ell}h_{k\ell}}{e_{ij}}\right).
\end{equation*}
Each of the above functions listed are the product of frobenian multiplicative functions of conductor $O(1)\abs{\theta_\textbf{A}}^{O(1)}$. Thus by \cite[Lem.~4.3]{TimSantens} these products are also frobenian multiplicative functions of conductor $O(1)\abs{\theta_\textbf{A}}^{O(1)}$. Furthermore, these frobenian multiplicative functions have mean $0$ when there is a variable linked to $x$ that is not equal to $1$, and hence every variable is frobenian (for $q_{\text{frob}}=O(1)\abs{\theta_\textbf{A}}^{O(1)}$). Therefore by \cite[Thm.~4.22]{TimSantens} we have
\begin{equation*}
S^U_\textbf{A}(\boldsymbol{\chi},T)=\mathop{\sum\sum\sum\sum\sum}_{\substack{(\textbf{u,hf,de})\in U \text{ satisfies } \star_{\epsilon}\\ \mu(\prod_iu_i\prod_{j\neq i}h_{ij}f_{ij}d_{ij}e_{ij})^2=1}}F(\textbf{u},\textbf{h},\textbf{f},\textbf{d},\textbf{e})+   O_{\epsilon}\left(\frac{T^\frac{3}{2}(\log T)^{\frac{5}{2}M_\textbf{A}}}{\abs{\theta_\textbf{A}}^{{\frac{1}{2}}-\epsilon}}\right).
\end{equation*}
We now consider the case when one of the variables is not equal to $1$ and compute its contribution to the overall sum. Let us fix such a variable and denote it by $x$. We can bound the part of $S^U_\textbf{A}(\boldsymbol{\chi},T)$ which is the sum over all variables with $x$ fixed above by
\begin{align*}
  \mathop{\sum_\textbf{u}\sum_\textbf{h}\sum_\textbf{f}\sum_\textbf{d}\sum_\textbf{e}}_{\substack{\abs{A_i}u_i^2\prod_{j\neq i}h_{ij}d_{ij}f_{ij}e_{ij}\leq T\\ y\leq e^{\left(\frac{3}{2}\log T\right)^\epsilon}, \text{ for } y \text{ linked to } x}}\prod_{i<j}\Big[\alpha(h_{ij};\theta_\textbf{A})\alpha(f_{ij};\theta_\textbf{A})\beta(d_{ij};\theta_\textbf{A})\beta(e_{ij};\theta_\textbf{A})\Big].  
\end{align*}
Summing over the $u_i$ and applying \cite[Prop.~4.12]{TimSantens} we may bound this sum above by
\begin{equation*}
\ll_{\epsilon}\frac{T^\frac{3}{2}(\log T)^{W}}{\abs{\theta_\textbf{A}}^{\frac{1}{2}-\epsilon}}\mathop{\sum}_{\substack{y \text{ linked to } x\\ y\leq e^{\left(\frac{3}{2}\log T\right)^\epsilon}}} \frac{1}{\prod_y y}\ll_{\epsilon}\frac{T^\frac{3}{2}(\log T)^{W+N\epsilon}}{\abs{\theta_\textbf{A}}^{\frac{1}{2}-(N+1)\epsilon}},
\end{equation*}
where $N=15$ is the total number of variables and $W$ is the sum of the means of the frobenian multiplicative functions corresponding to the variables not linked to $x$ (that is, the mean of $\alpha(h_{ij};\theta_\textbf{A}), \alpha(f_{ij};\theta_\textbf{A}), \beta(d_{ij};\theta_\textbf{A})$ or $\beta(e_{ij};\theta_\textbf{A})$ if $h_{ij}, f_{ij}, d_{ij}$ or $e_{ij}$ are not linked to $x$, respectively).\par
If $x=f_{ij}$, then $x$ is not linked to $f_{ij}, e_{ij}, d_{ij}$ and $h_{ij}$. Thus we have 
\begin{equation*}
    W=\frac{1}{4}+\frac{1}{4}+\frac{1}{8}+\frac{1}{8}=\frac{3}{4}=2M_\textbf{A},
\end{equation*}
when $\theta_\textbf{A}\not\in\pm\Q^{\times 2}$ and 
\begin{equation*}
    W=\frac{1}{4}+\frac{1}{4}+\frac{1}{4}+\frac{1}{4}=1=2M_\textbf{A},
\end{equation*}
when $\theta_\textbf{A}\in-\Q^{\times 2}$. Therefore the contribution of the sum when $f_{ij}$ is not equal to $1$ is less than
\begin{equation*}
 O_{\epsilon}\left(\frac{T^\frac{3}{2}(\log T)^{\frac{5}{2}M_\textbf{A}}}{\abs{\theta_\textbf{A}}^{{\frac{1}{2}}-\epsilon}}\right),   
\end{equation*}
and hence we may assume $f_{ij}=1$ for all $i\neq j$.\par
If $x=e_{ij}$, then $x$ is not linked to $f_{ij}, e_{ij}, d_{ij}$ and $h_{ij}$. Thus we have
\begin{equation*}
    W=\frac{1}{4}+\frac{1}{4}+\frac{1}{8}+\frac{1}{8}=\frac{3}{4}=2M_\textbf{A},
\end{equation*}
when  $\theta_\textbf{A}\not\in\pm\Q^{\times 2}$ and 
\begin{equation*}
    W=\frac{1}{4}+\frac{1}{4}+\frac{1}{4}+\frac{1}{4}=1=2M_\textbf{A},
\end{equation*}
when $\theta_\textbf{A}\in-\Q^{\times 2}$. Therefore the contribution of the sum when $e_{ij}$ is not equal to $1$ is less than 
\begin{equation*}
   O_{\epsilon}\left(\frac{T^\frac{3}{2}(\log T)^{\frac{5}{2}M_\textbf{A}}}{\abs{\theta_\textbf{A}}^{{\frac{1}{2}}-\epsilon}}\right), 
\end{equation*}
and hence we may assume $e_{ij}=1$ for all $i\neq j$. Thus the result follows.
\end{proof}
\begin{remark}
The power of $\log T$ in the error term can actually be chosen to be any value strictly greater than $2M_\textbf{A}$. Our choice of $\frac{5}{2}M_\textbf{A}$ was only chosen so that it does not interfere with the results in the following section. 
\end{remark}

\begin{proof}[Proof of Lemma~\ref{lem: bound for downward open sets for reduction}]
By substituting Lemma~\ref{lem: formula for downward-closed subset N^(loc,U)(T)} into \eqref{eq: N_A,M^loc(T) second simplified form (M condition removed)} for any downward-closed set $U$ we have
\begin{align}
\label{eq: N^(loc,U)(T) upper bound 1}
    \#\NAM^{loc,U}(T)\ll_{\epsilon} \mathop{\sum\sum\sum}_{\substack{(\textbf{u,h,d})\in U}} \prod_{i<j}\alpha(h_{ij};\theta_\textbf{A})\beta(d_{ij};\theta_\textbf{A})
     +O_{\epsilon}\left(\frac{T^\frac{3}{2}(\log T)^{\frac{5}{2}M_\textbf{A}}}{\abs{\theta_\textbf{A}}^{\frac{1}{2}-\epsilon}}\right).
\end{align}
We now deal with the two sets independently:\par
Let $U$ be the left-hand side set in \eqref{eq: first downward open set} (in Lemma~\ref{lem: bound for downward open sets for reduction}), then summing over $u_i$ we have that the right-side of \eqref{eq: N^(loc,U)(T) upper bound 1} is bounded by
\begin{equation}
\label{eq: case 1 for bounds for downward-closed subsets}
    \frac{T^{\frac{3}{2}}}{\abs{\theta_\textbf{A}}^\frac{1}{2}}\mathop{\sum}_{\substack{h_{ij}\leq (\log T^\frac{3}{2})^C}} \prod_{i<j} \frac{\alpha(h_{ij};\theta_\textbf{A})}{h_{ij}} \mathop{\sum}_{\substack{d_{ij}\leq T}}\prod_{i<j}\frac{\beta(d_{ij};\theta_\textbf{A})}{d_{ij}}
\end{equation}
Now by \cite[Thm.~4.12]{TimSantens} and partial summation we have
\begin{equation*}
    \mathop{\sum}_{\substack{h_{ij}\leq (\log T^\frac{3}{2})^C}} \prod_{i<j} \frac{\alpha(h_{ij};\theta_\textbf{A})}{h_{ij}}\ll_{\epsilon} \abs{\theta_\textbf{A}}^{3\epsilon}(\log (\log T^\frac{3}{2})^{C})^{W_1},
\end{equation*}
and
\begin{equation*}
   \mathop{\sum}_{\substack{d_{ij}\leq T}}\prod_{i<j}\frac{\beta(d_{ij};\theta_\textbf{A})}{d_{ij}}\ll_{\epsilon}\abs{\theta_\textbf{A}}^{3\epsilon} (\log T)^{W_2}, 
\end{equation*}
where $W_1= \frac{3}{4}$ and $W_2= \frac{3}{4}$ if $\theta_\textbf{A}\in-\Q^{\times 2}$ and $W_1=\frac{3}{8}$ and $W_2=\frac{3}{4}$ if $\theta_\textbf{A}\not\in\pm\Q^{\times 2}$. Therefore substituting into \eqref{eq: case 1 for bounds for downward-closed subsets} we have
\begin{equation*}
    \ll_{\epsilon} \abs{\theta_\textbf{A}}^{6\epsilon-\frac{1}{2}}T^\frac{3}{2}\log T^{W_2} (\log (\log T^\frac{3}{2})^{C})^{W_1}\ll_C \frac{T^\frac{3}{2}(\log T)^{\frac{5}{2}M_\textbf{A}}}{\abs{\theta_\textbf{A}}^{-6\epsilon+\frac{1}{2}}}.
\end{equation*}
Now let $U$ be the left-hand side set in \eqref{eq: second downward open set} (in Lemma~\ref{lem: bound for downward open sets for reduction}), then without loss of generality assume $i=0$. Summing over $u_0$, we may bound \eqref{eq: N^(loc,U)(T) upper bound 1} above by
\begin{equation*}
 \frac{T^\frac{1}{2}}{\abs{A_0}^\frac{1}{2}}\mathop{\sum}_{\substack{u_1,u_{2}\leq (\log T^\frac{3}{2})^C}} \mathop{\sum\sum}_{\substack{h_{ij},d_{ij}\\\abs{A_i}u_i^2\prod_{i<j}h_{ij}d_{ij}\leq T}} \alpha(h_{12};\theta_\textbf{A}) \beta(d_{12};\theta_\textbf{A}) \prod_{j\neq 0}\frac{\alpha(h_{0j};\theta_\textbf{A})}{h_{0j}^\frac{1}{2}} \frac{\beta(d_{0j};\theta_\textbf{A})}{d_{0j}^{\frac{1}{2}}}.
\end{equation*}
Now summing over $h_{01},h_{02}$ and using that $\alpha(h_{ij};\theta_\textbf{A})\leq 1$ we have
\begin{equation}
\label{eq: case 2 for bounds for downward-closed subsets}
 \frac{T^\frac{3}{2}}{\abs{\theta_\textbf{A}}^\frac{1}{2}}\mathop{\sum}_{\substack{u_1,u_{2}\leq (\log T^\frac{3}{2})^C}} \frac{1}{u_1u_2}\mathop{\sum\sum}_{\substack{h_{12},d_{ij}\\\abs{A_i}u_i^2\prod_{i<j}h_{ij}d_{ij}\leq T}} \frac{\alpha(h_{12};\theta_\textbf{A})}{h_{12}} \prod_{i<j} \frac{\beta(d_{ij};\theta_\textbf{A})}{d_{ij}}.
\end{equation}
As in the previous argument, by using \cite[Thm.~4.12]{TimSantens} and partial summation we have
\begin{equation*}
    \mathop{\sum}_{\substack{h_{12}\leq T}} \frac{\alpha(h_{12};\theta_\textbf{A})}{h_{12}}\ll_{\epsilon} \abs{\theta_\textbf{A}}^{\epsilon}(\log T)^{W_1'},
\end{equation*}
\begin{equation*}
   \mathop{\sum}_{\substack{d_{ij}\leq T}}\prod_{i<j}\frac{\beta(d_{ij};\theta_\textbf{A})}{d_{ij}}\ll_{\epsilon}\abs{\theta_\textbf{A}}^{3\epsilon} (\log T)^{W_2'}, 
\end{equation*}
and
\begin{equation*}
    \mathop{\sum}_{\substack{u_1,u_2\leq (\log T^\frac{3}{2})^C}} \frac{1}{u_1u_2}\ll (\log(\log T^\frac{3}{2})^C)^{2},
\end{equation*}
where $W_1'=\frac{1}{4}$ and $W_2'=\frac{3}{4}$ if $\theta_\textbf{A}\in-\Q^{\times 2}$ and $W_1'=\frac{1}{8}$ and $W_2'=\frac{3}{4}$ if $\theta_\textbf{A}\not\in\pm\Q^{\times 2}$. Therefore substituting into \eqref{eq: case 2 for bounds for downward-closed subsets} we have
\begin{equation*}
    \ll_{\epsilon} \abs{\theta_\textbf{A}}^{4\epsilon-\frac{1}{2}}T^\frac{3}{2}\log T^{W_1'+W_2'} (\log(\log T^\frac{3}{2})^C)^{2}\ll_C \frac{T^\frac{3}{2}(\log T)^{\frac{5}{2}M_\textbf{A}}}{\abs{\theta_\textbf{A}}^{-4\epsilon+\frac{1}{2}}}.
\end{equation*}

\end{proof}
\subsubsection{Simplifying $S^U_\textbf{A}(\boldsymbol{\chi},T)$}

\begin{lemma}
\label{lem: defining lambda function}
Let $x\in \Z\setminus\{0\}$ such that $x$ is not a square. Let 
\begin{equation*}
    \lambda(t;x):=\mu(t)^2\mathop{\sum}_{\substack{hd=t\\ (x \prod_{p\in S}p,t)=1}}\alpha(h;x)\beta(d;x),
\end{equation*}
where $(y,z)$ denotes the greatest common divisor of $y$ and $z$, then $\lambda(t:x)$ is a $S\cup \{p:p\mid x\}$-frobenian multiplicative function of mean $\frac{1}{2}$ when $x \in -\Q^{\times 2}$ and mean $\frac{3}{8}$ when $x\not\in\pm \Q^{\times 2}$.
\end{lemma}
\begin{proof}
The proof follows from the discussion after the proof of \cite[Lem.~5.4]{TimSantens} after appropriately changing the functions.    
\end{proof}
Now consider $S_\textbf{A}(\boldsymbol{\chi},T)$. As in Remark~\ref{rem: extending chi_i to all of Z/8m_AZ} we have that $\chi_i(x)=0$ when $(m_\textbf{A},x)\neq 1$, and hence substituting Lemma~\ref{lem: defining lambda function} into Lemma~\ref{lem: formula for downward-closed subset N^(loc,U)(T)} we have
\begin{equation}
\label{eq: S_A(x,T) with only mu left to deal with}
\begin{split}
     S_\textbf{A}(\boldsymbol{\chi},T)=&\mathop{\sum_\textbf{u}\sum_\textbf{t}}_{\substack{\abs{A_i}u_i^2\prod_{j\neq i} t_{ij}\leq T}}\mu(\prod_iu_i\prod_{j\neq i}t_{ij})^2\prod_i \chi_i(u_i^2)\prod_{j\neq i}\chi_i(\chi_{j}(t_{ij}))\lambda(t_{ij};\theta_\textbf{A})\\
     &+O_{\epsilon}\left(\frac{T^\frac{3}{2}(\log T)^{\frac{5}{2}M_\textbf{A}}}{\abs{\theta_\textbf{A}}^{\frac{1}{2}-\epsilon}}\right).
\end{split} 
\end{equation}

\subsubsection{Dealing with $\mu(-)$ condition on $\NAM^{loc}(T)$}

\begin{lemma}
\label{lem: dealing with mu(-) in terms of kappa}
There exists a $6$-variable multiplicative function $\kappa$ such that
\begin{equation*} \mu(\prod_iu_i\prod_{j\neq i}t_{ij})^2=\mathop{\sum}_{\substack{\textbf{f}\\  f_{ij}\mid t_{ij}}} \mathop{\sum}_{\substack{ \textbf{g}\\ g_{i}\mid u_i}}\kappa(\textbf{f},\textbf{g}).
\end{equation*}
\end{lemma}
\begin{proof}
This equality follows from M\"obius inversion. Explicitly, let $\kappa$ denote the $6$-variable multiplicative function with the Dirichlet series
\begin{align*}
    F(\textbf{s},\textbf{s'})=\sum_{\textbf{u}}\sum_\textbf{t}\frac{\mu(\prod_iu_i\prod_{j\neq i}t_{ij})^2}{\prod_iu_i^{s_i}\prod_{j\neq i}t_{ij}^{s_{ij}'}}\prod_i\zeta(s_i)^{-1}\prod_{j\neq i}\zeta(s_{ij}')^{-1}.
\end{align*}  
For $s_i,s'_{ij}>\frac{1}{2}$ we have
\begin{equation*}
    F(\textbf{s},\textbf{s'})=\prod_p\left(1+\sum_ip^{-s_i}+\sum_{i<j}p^{-s_{ij}'}\right)\prod_i\left(1-p^{-s_i}\right)\prod_{j\neq i}\left(1-p^{-s_{ij}'}\right).
\end{equation*}
Expanding out the terms of the product over $p$ we find that we can bound these inner terms by $1$, and hence the product converge absolutely. Thus this series converges absolutely when $s_i,s'_{ij}>\frac{1}{2}$ for all $i\neq j$. Therefore by M\"obius inversion we have
\begin{equation*}
\mu(\prod_iu_i\prod_{j\neq i}t_{ij})^2=\mathop{\sum}_{\substack{\textbf{f}\\  f_{ij}\mid t_{ij}}} \mathop{\sum}_{\substack{ \textbf{g}\\ g_{i}\mid u_i}}\kappa(\textbf{f},\textbf{g}).\qedhere
\end{equation*}
\end{proof}
Now rewriting $u_i=g_iu_i$ and $t_{ij}=f_{ij}t_{ij}$ for all $i\neq j$, and substituting Lemma~\ref{lem: dealing with mu(-) in terms of kappa} into \eqref{eq: S_A(x,T) with only mu left to deal with} we have
\begin{equation*}
\label{eq: S_A(x,T) with only kappa left to deal with}
\begin{split}
     S_\textbf{A}(\boldsymbol{\chi},T)=&\mathop{\sum_\textbf{f}\sum_\textbf{g}\sum_\textbf{u}\sum_\textbf{t}}_{\substack{\abs{A_i}g_i^2u_i^2\prod_{j\neq i} f_{ij}t_{ij}\leq T\\ (f_{ij},t_{ij})=1}}\kappa(\textbf{f},\textbf{g})\prod_i \chi_i(g_i^2u_i^2)\prod_{j\neq i}\chi_i(\chi_{j}(f_{ij}t_{ij}))\lambda(f_{ij};\theta_\textbf{A})\lambda(t_{ij};\theta_\textbf{A})\\
     &+O_{\epsilon}\left(\frac{T^\frac{3}{2}(\log T)^{\frac{5}{2}M_\textbf{A}}}{\abs{\theta_\textbf{A}}^{\frac{1}{2}-\epsilon}}\right).
\end{split} 
\end{equation*}

\subsubsection{Dealing with $\chi_i(-)$ condition on $S_\textbf{A}(\chi,T)$}
\begin{lemma}
\label{lem: chi is a 2-torsion element reduction}
Let $\boldsymbol{\chi}=(\chi_0,\chi_1,\chi_2)\in\left(\Gamma_\textbf{A}^{\vee} \right)^3$. If there exists $\chi\in \Gamma^{\vee}_\textbf{A}[2]$ such that $\chi_i=\chi$ for all $i$, then we have
\begin{align*}
 S_\textbf{A}(\boldsymbol{\chi},T)=S_\textbf{A}(\boldsymbol{1},T)= &\mathop{\sum_\textbf{f}\sum_\textbf{g}\sum_\textbf{u}\sum_\textbf{t}}_{\substack{\abs{A_i}g_i^2u_i^2\prod_{j\neq i}f_{ij}t_{ij}\leq T\\ (m_\textbf{A},\prod_iu_ig_i)=(f_{ij},t_{ij})=1}}\kappa(\textbf{f},\textbf{g})\prod_{i<j}\lambda(f_{ij};\theta_\textbf{A})\lambda(t_{ij};\theta_\textbf{A})\\
 &+ O_{\epsilon}\left(\frac{T^\frac{3}{2}\left(\log T\right)^{\frac{5}{2}M_\textbf{A}}}{\abs{\theta_\textbf{A}}^{\frac{1}{2}-\epsilon}}\right).
\end{align*} 
Otherwise we have
\begin{equation*}
    S_\textbf{A}(\boldsymbol{\chi},T)=O_{\epsilon}\left(\frac{T^\frac{3}{2}\left(\log T\right)^{\frac{5}{2}M_\textbf{A}}}{\abs{\theta_\textbf{A}}^{\frac{1}{2}-\epsilon}}\right).
\end{equation*}
\end{lemma}
\begin{proof}
The proof follows from the proof of \cite[Lem.~5.11]{TimSantens} after adjusting for the different variables.    
\end{proof}
Now substituting Lemma~\ref{lem: chi is a 2-torsion element reduction} into \eqref{eq: S_A(X,T) first simplified form (M condition removed)} we have
\begin{align}
    \#\NAM^{loc}(T)&= \frac{1}{\abs{\Gamma_\textbf{A}}^3}\mathop{\sum}_{\substack{\chi\in\Gamma^{\vee}_\textbf{A}[2]}}S_\textbf{A}(\textbf{1},T)\prod_{i}\overline{\chi}(M_i)+O_{\epsilon}\left(\frac{T^\frac{3}{2}(\log T)^{\frac{5}{2}M_\textbf{A}}}{\abs{\theta_\textbf{A}}^{\frac{1}{2}-\epsilon}}\right)\\
    &= \label{eq: final form of N^loc(T) in terms of S(1,T)} \frac{\abs{\Gamma^{\vee}_\textbf{A}[2]}}{\abs{\Gamma_\textbf{A}}^3}S_\textbf{A}(\textbf{1},T)+O_{\epsilon}\left(\frac{T^\frac{3}{2}(\log T)^{\frac{5}{2}M_\textbf{A}}}{\abs{\theta_\textbf{A}}^{\frac{1}{2}-\epsilon}}\right),
\end{align}
where the second equality follows from $M_0M_1M_2$ being a square in $\Gamma_\textbf{A}$ for all $\textbf{M}\in \Psi_\textbf{A}$.

\subsection{The main term}
\label{subsec: The main term}
Before providing an asymptotic formula for $S_\textbf{A}(\textbf{1},T)$, we give another result we will use within the proof:
\begin{lemma}
\label{lem: partial sum of lambdas}
For any $T,f\geq 1$, there exists a constant $K_{\textbf{A},f}\leq K_{\textbf{A},1}\ll_\epsilon\abs{\theta}^\epsilon$ such that
 \begin{align*}
     \mathop{\sum}_{\substack{t\leq T\\ (t,f)=1}} \frac{\lambda(t;\theta_\textbf{A})}{t}=K_{\textbf{A},f}(\log T)^{M_{\textbf{A}}}+O_\epsilon\left((\abs{\theta_\textbf{A}}f)^\epsilon\right).
 \end{align*}
\end{lemma}
\begin{proof}
This follows from the discussion between the start of \S5.4 and equation 5.21 in \cite{TimSantens}. 
\end{proof}
\begin{lemma}
\label{lem: assymptotic for S_A(1,T)}
For all $(\textbf{A},\textbf{M})\in$ \textbf{$\Phi(T)$}$\times$$\Psi_\textbf{A}$ there exists a constant $0<Q'_\textbf{A}\ll_{\epsilon} \abs{\theta_\textbf{A}}^{\epsilon}$ for all $\epsilon>0$ such that for $T>2$ we have
    \begin{equation*}
    S_\textbf{A}(\textbf{1},T)=\left(Q'_\textbf{A}+O_{\epsilon}\left(\frac{\abs{\theta_\textbf{A}}^{\epsilon}}{\left(\log T\right)^{\frac{6}{5}M_\textbf{A}}}\right)\right)\frac{T^\frac{3}{2}\left(\log T\right)^{3M_\textbf{A}}}{\abs{\theta_\textbf{A}}^{\frac{1}{2}}}.
    \end{equation*}
\end{lemma}
\begin{proof}
We adapt the proof of \cite[Lem.~5.12]{TimSantens}. By Lemma~\ref{lem: chi is a 2-torsion element reduction} we have that $S_\textbf{A}(\boldsymbol{1},T)$ is equal to
\begin{align}
\label{eq: main term start point} 
\mathop{\sum_\textbf{f}\sum_\textbf{g}\sum_\textbf{u}\sum_\textbf{t}}_{\substack{\abs{A_i}g_i^2u_i^2\prod_{j\neq i}f_{ij}t_{ij}\leq T\\ (m_\textbf{A},\prod_iu_ig_i)=(f_{ij},t_{ij})=1}}\kappa(\textbf{f},\textbf{g})\prod_{i<j}\lambda(f_{ij};\theta_\textbf{A})\lambda(t_{ij};\theta_\textbf{A})
 + O_{\epsilon}\left(\frac{T^\frac{3}{2}\left(\log T\right)^{\frac{5}{2}M_\textbf{A}}}{\abs{\theta_\textbf{A}}^{\frac{1}{2}-\epsilon}}\right).  
\end{align}
Since the summand does not depend on $u_i$, the contribution of the sum over any $u_i$ is
\begin{equation*}
   \frac{\phi(m_\textbf{A})T^\frac{1}{2}}{m_\textbf{A}g_i(\abs{A_i}\prod_{j\neq i}f_{ij}t_{ij})^{\frac{1}{2}}}+O(\tau_2(m_\textbf{A})),
\end{equation*}
where $\phi$ denotes Euler's totient function. Summing over $u_2, u_1$, and then $u_0$, the error term that arises is bounded by
\begin{align*}
    \ll  \frac{\tau_2(m_\textbf{A})T}{\abs{A_1A_2}^{\frac{1}{2}}}\mathop{\sum_\textbf{f}\sum_\textbf{g}\sum_{\textbf{t}}}_{\substack{\abs{A_i}g_i^2\prod_{j\neq i}f_{ij}t_{ij}\leq T\\ (m_\textbf{A},\prod_iu_ig_i)=(f_{ij},t_{ij})=1}} \frac{\abs{\kappa(\textbf{f},\textbf{g})}g_0\prod_{k\neq 0}\sqrt{f_{0k}t_{0k}}\prod_{i<j}\lambda(f_{ij};\theta_\textbf{A})\lambda(t_{ij};\theta_\textbf{A})}{\prod_ig_i\prod_{j\neq i}f_{ij}t_{ij}}.
\end{align*}
Summing over $t_{01}\leq T/g_0^2t_{02}t_{03}\prod_{j\neq 0}f_{0j}$ we obtain the upper bound
\begin{align*}
\label{eq: error term for summing over u_i in assymtpotic formula}
    \ll  \frac{\tau_2(m_\textbf{A})T^\frac{3}{2}}{\abs{\theta_\textbf{A}}^{\frac{1}{2}}}\mathop{\sum_\textbf{f}\sum_\textbf{g}\sum_{\textbf{t}\setminus\{t_{01}\}}}_{\substack{t_{ij}\leq T}} \frac{\abs{\kappa(\textbf{f},\textbf{g})}\prod_{i<j}\lambda(f_{ij};\theta_\textbf{A})\lambda(t_{ij};\theta_\textbf{A})}{\prod_ig_i\prod_{j\neq i}f_{ij}t_{ij}}.
\end{align*}
Summing over $t_{02},t_{12}\leq T$, and using Lemma~\ref{lem: partial sum of lambdas} we have obtain the upper bound
\begin{align*}
    \ll_\epsilon \frac{T^\frac{3}{2}(\log T)^{3M_\textbf{A}}}{\abs{\theta_\textbf{A}}^{\frac{1}{2}-\epsilon}}\mathop{\sum_\textbf{f}\sum_\textbf{g}}\frac{\abs{\kappa(\textbf{f},\textbf{g})}}{\prod_ig_i\prod_{j\neq i}f_{ij}},
\end{align*}
where we have used the divisor bound $\tau_2(m_\textbf{A})\ll_\epsilon\abs{\theta_\textbf{A}}^\epsilon$ and we have rescaled $\epsilon>0$. Now by Lemma~\ref{lem: dealing with mu(-) in terms of kappa} we have that the sum over \textbf{f} and \textbf{g} is bounded by the sum of the absolute values of the summands defining $F(\textbf{1},\textbf{1})$. In particular since $F$ converges absolutely at $(\textbf{1},\textbf{1})$ this sum converges, and hence the error term is bounded by
\begin{align*}
    \ll_\epsilon \frac{T^\frac{3}{2}(\log T)^{3M_\textbf{A}}}{\abs{\theta_\textbf{A}}^{\frac{1}{2}-\epsilon}}.
\end{align*}
In particular summing over all $u_i$ we obtain
\begin{equation}
    \begin{split}
    \label{eq: main term first reduction}
        S_\textbf{A}(\textbf{1},T)=&\frac{\phi(m_\textbf{A})^3T^\frac{3}{2}}{m_\textbf{A}^3\abs{\theta_\textbf{A}}^{\frac{1}{2}}}\mathop{\sum_\textbf{f}\sum_\textbf{g}\sum_\textbf{t}}_{\substack{A_ig_i^2\prod_{j\neq i} f_{ij}t_{ij}\leq T\\
    (m_\textbf{A},\prod_ig_i)=(f_{ij},t_{ij})=1}}\frac{\kappa(\textbf{f},\textbf{g})\prod_{i<j}\lambda(f_{ij};\theta_\textbf{A})\lambda(t_{ij};\theta_\textbf{A})}{\prod_ig_i\prod_{j\neq i}f_{ij}t_{ij}}\\
    &+O_\epsilon\left(\frac{T^\frac{3}{2}(\log T)^{3M_\textbf{A}}}{\abs{\theta_\textbf{A}}^{\frac{1}{2}-\epsilon}}\right).
    \end{split}
\end{equation}
Now let
\begin{align*}
    \operatorname{R}_\textbf{A}(T;\textbf{f},\textbf{g}):=\mathop{\sum_\textbf{t}}_{\substack{A_ig_i^2\prod_{j\neq i} f_{ij}t_{ij}\leq T\\
    (m_\textbf{A},\prod_ig_i)=(f_{ij},t_{ij})=1}}\prod_{i<j}\frac{\lambda(t_{ij};\theta_\textbf{A})}{f_{ij}t_{ij}}, \text{ and }  \operatorname{R}'_\textbf{A}(T;\textbf{f}):=\mathop{\sum_\textbf{t}}_{\substack{\prod_{i<j}t_{ij}\leq T\\ (f_{ij},t_{ij})=1}}\prod_{i<j}\frac{\lambda(t_{ij};\theta_\textbf{A})}{t_{ij}}.
\end{align*}
Clearly that the difference of these sums is the subsum of $\operatorname{R}'_\textbf{A}(T;\textbf{f})$ defined by the additional condition that for each $i\in\{0,1,2\}$ we have
\begin{align*}
    \frac{T}{\abs{A_i}g_i^2\prod_{j\neq i}f_{ij}}\leq\prod_{j\neq i}t_{ij}\leq T.
\end{align*}
To evaluate this difference of sums we first sum over $t_{01}$ to obtain the upper bound
\begin{align*}
    \ll\log\left(\abs{A_0}g_0^2\prod_{j\neq 0}f_{0j}\right)\mathop{\sum_{\textbf{t}\setminus\{t_{01}\}}}_{\substack{ t_{ij}\leq T\\(f_{ij},t_{ij})=1}}\prod_{i<j}\frac{\lambda(t_{ij};\theta_\textbf{A})}{t_{ij}}.
\end{align*}
Now summing over $t_{02},t_{12}\leq T$ and using Lemma~\ref{lem: partial sum of lambdas} we obtain the upper bound
\begin{align*}
\ll_{\epsilon}\log\left(\abs{\theta_\textbf{A}}\prod_ig_i^2\prod_{j\neq i}f_{ij}\right)\abs{\theta_\textbf{A}}^{2\epsilon}(\log T)^{2M_\textbf{A}}\ll_{\epsilon} \left(\abs{\theta_\textbf{A}}\prod_ig_i\prod_{j\neq i}f_{ij}\right)^{3\epsilon}(\log T)^{2 M_\textbf{A}}.
\end{align*}
In particular we have
\begin{align}
\label{eq: main term evluating R_A(T;f,g)}
   \operatorname{R}_\textbf{A}(T;\textbf{f},\textbf{g})=\operatorname{R}'_\textbf{A}(T;\textbf{f})+O_{\epsilon}\left(\left(\abs{\theta_\textbf{A}}\prod_ig_i\prod_{j\neq i}f_{ij}\right)^{3\epsilon}(\log T)^{2 M_\textbf{A}}\right).
\end{align}
Now consider $\operatorname{R}'_\textbf{A}(T;\textbf{f})$. Let $\hat{g_{ij}},\hat{h_{ij}},\hat{f_{ij}}:\RR_{\geq0}\rightarrow\RR_{\geq0}$ be the functions given by
 \[\hat{g_{ij}}(x):=
        \begin{cases}
        K_{\textbf{A},f_{ij}}\frac{d}{dx}(\log x)^{M_\textbf{A}}, &\text{ if } x\geq 1,\\
        0, &\text{ otherwise};
        \end{cases}\]
 \[\hat{h_{ij}}(x):=
        \begin{cases}
         O_{\epsilon}((\abs{\theta_\textbf{A}}f_{ij})^{\epsilon}), &\text{ if } \frac{1}{2}\leq x\leq 1,\\
        0, &\text{ otherwise};
        \end{cases}\]
 \[\hat{f_{ij}}(x):=
        \begin{cases}
        \frac{\lambda(x;\theta_\textbf{A})}{x}, &\text{ if } x\geq 1,\\
        0, &\text{ otherwise}.
        \end{cases}\]
By Lemma~\ref{lem: partial sum of lambdas} we may apply \cite[Lem.~4.17]{TimSantens} to $\hat{g_{ij}},\hat{h_{ij}}$ and $\hat{f_{ij}}$, and hence obtain
\begin{align*}
    \operatorname{R}'_\textbf{A}(T;\textbf{f})=&\int^{\prod_{j\neq i}t_{ij}\leq T}_{1\leq t_{ij}}\prod_{i<j}K_{\textbf{A},f_{ij}}\frac{d}{dt_{ij}}(\log t_{ij})^{M_\textbf{A}}dt_{ij}\\
    &+ O_\epsilon\left(\sum_{J\subsetneq\{0,1,2\}}\int^{\prod_{j\neq i}t_{ij}\leq T}_{0\leq t_{ij}} \prod_{i,j\in J} \hat{g_{ij}}(t_{ij}) dt_{ij}\prod_{m,n\in\{0,1,2\}\setminus J}\hat{h_{ij}}(t_{ij})dt_{mn}\right).
\end{align*}
The integral within the error term can be bounded above by
\begin{equation*}
    \prod_{i,j\in J} K_{\textbf{A},f_{ij}}\int^{4T}_1 \frac{d}{dt_{ij}}(\log t_{ij})^{M_\textbf{A}} dt_{ij}\prod_{m,n\in\{0,1,2\}\setminus J}(\abs{\theta_\textbf{A}}f_{mn})^{\epsilon}\int^1_{\frac{1}{2}}dt_{mn}.
\end{equation*}
Since $\abs{J}\leq 2$ for any choice of $J$, we can bound this by
\begin{align*}
\ll_{\epsilon}O_\epsilon\left(\left(\abs{\theta_\textbf{A}}\prod_{i<j}f_{ij}\right)^{3\epsilon}\left(\log T\right)^{2M_\textbf{A}}\right).
\end{align*}
Thus we have
\begin{equation}
    \begin{split}
        \label{eq: main term evaluated R_A(t;f)}
    \operatorname{R}'_\textbf{A}(T;\textbf{f})=&\int^{\prod_{j\neq i}t_{ij}\leq T}_{1\leq t_{ij}}\prod_{i<j}K_{\textbf{A},f_{ij}}\frac{d}{dt_{ij}}(\log t_{ij})^{M_\textbf{A}}dt_{ij}\\  &+O_{\epsilon}\left(\left(\abs{\theta_\textbf{A}}\prod_{i<j}f_{ij}\right)^{3\epsilon}\left(\log T\right)^{2M_\textbf{A}}\right).
    \end{split}
\end{equation}
Now by making the change of variables $s_{ij}=\left(\frac{\log t_{ij}}{\log T}\right)^{M_\textbf{A}}$, we have that the main term of \eqref{eq: main term evaluated R_A(t;f)} becomes
\begin{align}
\label{eq: main term remove integral}
    \prod_{i<j}K_{\textbf{A},f_{ij}}(\log T)^{3M_\textbf{A}}\int^{\sum_{j\neq i} (s_{ij})^{\frac{1}{M_\textbf{A}}}\leq 1}_{0\leq s_{ij}}\prod_{i<j}ds_{ij}.
\end{align}
In particular the integral is equal to some constant $0<L_\textbf{A}\leq 1$. Therefore substituting \eqref{eq: main term remove integral} and \eqref{eq: main term evaluated R_A(t;f)} into \eqref{eq: main term evluating R_A(T;f,g)} we have
\begin{align}
\label{eq: MAIN TERM final eval for R_A(T;f,g)}
    \operatorname{R}_\textbf{A}(T;\textbf{f},\textbf{g})=L_\textbf{A}\prod_{i<j}K_{\textbf{A},f_{ij}}(\log T)^{3M_\textbf{A}}+O_{\epsilon}\left(\left(\abs{\theta_\textbf{A}}\prod_ig_i\prod_{j\neq i}f_{ij}\right)^{3\epsilon}(\log T)^{2 M_\textbf{A}}\right).
\end{align}
Furthermore, substituting \eqref{eq: MAIN TERM final eval for R_A(T;f,g)} into \eqref{eq: main term first reduction} we have that $S_\textbf{A}(\textbf{1},T)$ is equal to
\begin{equation}
\begin{split}
\label{eq: main term first S_A(1,T) assymtptotic}
     \frac{\phi(m_\textbf{A})^3L_\textbf{A}T^\frac{3}{2}(\log T)^{3M_\textbf{A}}}{m_\textbf{A}^3\abs{\theta_\textbf{A}}^{\frac{1}{2}}}\mathop{\sum_\textbf{f}\sum_\textbf{g}}_{\substack{A_ig_i^2\prod_{j\neq i} f_{ij}\leq T\\(m_\textbf{A},\prod_ig_i)=1}}\frac{\kappa(\textbf{f},\textbf{g})\prod_{i<j}K_{\textbf{A},f_{ij}}\lambda(f_{ij};\theta_\textbf{A})}{\prod_ig_i\prod_{j\neq i}f_{ij}}\\
    +O_{\epsilon}\left( \frac{T^\frac{3}{2}(\log T)^{3M_\textbf{A}}}{\abs{\theta_\textbf{A}}^{\frac{1}{2}-\epsilon}} + \frac{T^\frac{3}{2}\left(\log T\right)^{2M_\textbf{A}}}{\abs{\theta_\textbf{A}}^{\frac{1}{2}-\epsilon}}\mathop{\sum_\textbf{f}\sum_\textbf{g}}_{\substack{A_ig_i^2\prod_{j\neq i} f_{ij}\leq T\\(m_\textbf{A},\prod_ig_i)=1}}\frac{\kappa(\textbf{f},\textbf{g})\prod_{i<j}K_{\textbf{A},f_{ij}}\lambda(f_{ij};\theta_\textbf{A})}{\prod_ig_i^{1-\epsilon}\prod_{j\neq i}f_{ij}^{1-\epsilon}}\right).
\end{split}
\end{equation}
Now consider the sum
\begin{align}
\label{eq: sum over f and g}
\mathop{\sum_\textbf{f}\sum_\textbf{g}}_{\substack{A_ig_i^2\prod_{j\neq i} f_{ij}\leq T\\(m_\textbf{A},\prod_ig_i)=1}}\frac{\kappa(\textbf{f},\textbf{g})\prod_{i<j}K_{\textbf{A},f_{ij}}\lambda(f_{ij};\theta_\textbf{A})}{\prod_ig_i^{1-\epsilon}\prod_{j\neq i}f_{ij}^{1-\epsilon}}.
\end{align}
By Lemma~\ref{lem: partial sum of lambdas} we have $K_{\textbf{A},f_{ij}}\leq K_{\textbf{A},1}\ll_\epsilon \abs{\theta_\textbf{A}}^\epsilon$, and hence this is bounded above by
\begin{align*}
\ll_\epsilon\abs{\theta_\textbf{A}}^{3\epsilon}\mathop{\sum_\textbf{f}\sum_\textbf{g}}\frac{\abs{\kappa(\textbf{f},\textbf{g})}}{\prod_ig_i^{1-\epsilon}\prod_{j\neq i}f_{ij}^{1-\epsilon}}.
\end{align*}
The above series is the sum of the absolute values of the summands defining $F(\textbf{1}-\boldsymbol{\epsilon},\textbf{1}-\boldsymbol{\epsilon})$, and hence converges for sufficiently small $\epsilon>0$. Since the sum in the leading term of \eqref{eq: main term first S_A(1,T) assymtptotic} is clearly bounded above by \eqref{eq: sum over f and g} it also converges, and hence it is equal to some constant $0<F_\textbf{A}\ll_\epsilon\abs{\theta_\textbf{A}}^\epsilon$. Therefore substituting this constant into \eqref{eq: main term first S_A(1,T) assymtptotic} we obtain
\begin{align*}
  \frac{\phi(m_\textbf{A})^3T^\frac{3}{2}(\log T)^{3M_\textbf{A}}}{m_\textbf{A}^3\abs{\theta_\textbf{A}}^{\frac{1}{2}}}F_\textbf{A}L_\textbf{A}
    +O_{\epsilon}\left(\frac{T^\frac{3}{2}\left(\log T\right)^{\frac{5}{2}M_\textbf{A}}}{\abs{\theta_\textbf{A}}^{\frac{1}{2}-\epsilon}}\right).    
\end{align*}
Letting $Q'_\textbf{A}:=\phi(m_\textbf{A})^3F_\textbf{A}L_\textbf{A}/m_\textbf{A}^3$ we obtain the result.
\end{proof} 
\begin{proof}[Proof of Lemma~\ref{lem: assymptotic formula for local set}]
Substituting Lemma~\ref{lem: assymptotic for S_A(1,T)} into \eqref{eq: final form of N^loc(T) in terms of S(1,T)} we obtain the result.    
\end{proof}

\section{Uniform formula}
\label{sec: Uniform Formula}
In the paper \cite{On_the_arithmetic_of_del_Pezzo_surfaces_of_degree_2}, the authors were able to calculate $\Br S_\textbf{a}/ \Br \Q$ for all choices of $\textbf{a}\in(\Q^{\times})^3$. In their method, they utilised the fact that $H^3(\Q,\overline{\Q}^*)=0$ and the exact sequence below which arises from the Hochschild-Serre spectral sequence
\begin{equation*}
\begin{split}
  0\rightarrow   \Pic S_\textbf{a}\rightarrow (\Pic S_{\textbf{a},\overline{\Q}})^G \rightarrow\Br\Q\rightarrow \Br S_\textbf{a} \rightarrow \operatorname{H}^1(\Q, \Pic S_{\textbf{a},\overline{\Q}})\xrightarrow[]{d^{1,1}_S} \operatorname{H}^3(\Q,\overline{\Q}^*),
\end{split}
\end{equation*}
where $\overline{\Q}$ denotes an algebraic closure of $\Q$ and $G:=\Gal(\overline{\Q}/\Q)$. Let $a_0,a_1,a_2$ be algebraically independent transcendental elements over $\Q$. In this section we consider the surface given by
\begin{equation*}
    \mathcal{S}:a_0x_0^4+a_1x_1^4+a_2x_2^4=w^2\subseteq \P_k(1,1,1,2),
\end{equation*}
defined over $k:=\Q(a_0,a_1,a_2)$. By \cite[Lem.~3.3]{UematsuTetsuya} we have that $\operatorname{H}^3(k,\overline{k}^*)$ is non-trivial, and hence the same method cannot be used to calculate $\Br \mathcal{S}/ \Br k$. Instead, we will prove the last differential map in the above exact sequence is injective and hence obtain Theorem~\ref{thm: BrS/Brk=0}:
\begin{proof}[Proof of Theorem~\ref{thm: BrS/Brk=0}]
We adapt the methods used in \S6 of \cite{TimSantens} to prove the differential map is injective. Firstly, we can calculate $\operatorname{H}^1(k,\Pic \mathcal{S}_{\overline{k}})$ in exactly the same manner as $\operatorname{H}^1(\Q,\Pic S_{\overline{\Q}})$ was calculated for the generic case in \cite[Prop.~6]{On_the_arithmetic_of_del_Pezzo_surfaces_of_degree_2}. In particular we have $\operatorname{H}^1(k,\Pic \mathcal{S}_{\overline{k}})=\Z/2\Z$.\par
Let $U$ be the affine $k$-variety defined by the equation
\begin{equation*}
    a_0x_0^2+a_1x_1^2+a_2x_2^2-1=0,
\end{equation*}
and let $X\subseteq \P_k^3$ denote the compactification of $U$ given by
\begin{equation*}
   a_0x_0^2+a_1x_1^2+a_2x_2^2-t^2=0.
\end{equation*}
Consider the affine open $V:=\{\omega\neq 0\}\subseteq \mathcal{S}$. Then we get a morphism 
\begin{equation*}
   f:V\rightarrow U;[x_0:x_1:x_2:\omega]\mapsto [x_0^2:x_1^2:x_2^2:\omega], 
\end{equation*}
where we are viewing $U$ inside $X$. Therefore by the functoriality of the Hochschild-Serre spectral sequence we obtain a commutative diagram

 \[ \begin{tikzcd} \label{dia: commutative diagram from spectral sequence}
\operatorname{H}^1(k, \Pic \mathcal{S}_{\overline{k}}) \arrow{r}{\res{\cdot}{V}} \arrow[swap]{dr}{d^{1.1}_{\mathcal{S}}} & \operatorname{H}^1(k,\Pic V_{\overline{k}}) \arrow{d}{d^{1,1}_V}& \operatorname{H}^1(k,\Pic U_{\overline{k}}) \arrow[swap]{l}{f^*} \arrow{dl}{d^{1,1}_U}\\
& \operatorname{H}^3(k,\overline{k}^*)
\end{tikzcd}
\]
where by abuse of notation $\res{\cdot}{V}$ and $f^*$ denote the maps on the first cohomology induced by the maps $\res{\cdot}{V}:  \Pic \mathcal{S}_{\overline{k}} \rightarrow \Pic V_{\overline{k}}$ and $f^*:\Pic U_{\overline{k}}\rightarrow \Pic V_{\overline{k}}$, respectively. Let $\gamma:= \sqrt{-\frac{a_0}{a_1a_2}}$. From the proof of \cite[Prop.~2.2]{UematsuTetsuya} we know that 
\begin{equation*}
\operatorname{H}^1(k,\Pic U_{\Bar{k}})\cong \operatorname{H}^1(\Gal(k(\gamma)/k),\Pic U_{\Bar{k}})\cong \mathbb{Z}/2\mathbb{Z},
\end{equation*}
where the first isomorphism is the inflation map. Furthermore, by \cite[Thm.~3.1]{UematsuTetsuya} we know that $d^{1,1}_U$ is injective. Thus if we can find $\psi\in \operatorname{H}^1(k,\Pic\mathcal{S}_{\Bar{k}})$ such that $\res{\psi}{V}=f^*\phi$ for some generator $\phi\in \operatorname{H}^1(k,\Pic U_{\overline{k}})$, then by injectivity we have
\begin{equation*}
    d^{1,1}_\mathcal{S}(\psi)=d^{1,1}_V(\res{\psi}{V})=d^{1,1}_V(f^*\phi)=d^{1,1}_U(\phi)\neq0.
\end{equation*}
In particular, since $\operatorname{H}^1(k,\Pic \mathcal{S}_{\overline{k}})=\Z/2\Z$ we would have that $d^{1,1}_\mathcal{S}$ is injective. Therefore to prove the result it suffices to find such $\psi$ and $\phi$.\par
Let $\alpha:=\sqrt{-\frac{a_1}{a_0}}$ and $\beta:=\alpha\gamma$. Let $\Gal(k(\gamma)/k) = <\sigma>$ and 
\begin{equation*}
    L_1:= \{x_0-\alpha x_1 =x_2-\beta t=0\}\subseteq X_{\overline{k}}.
\end{equation*}
By \cite[Cor.~2.3]{UematsuTetsuya} the class of the $1$-cocycle $\phi:\Gal(k(\gamma)/k)\rightarrow\Pic U_{\overline{k}}$ defined by
\begin{equation*}
   \phi(1)=[0], \hspace{5mm} \phi(\sigma)=[L_1], 
\end{equation*}
generates $\operatorname{H}^1(\Gal(k(\gamma)/k),\Pic U_{\overline{k}})$, where by abuse of notation $[L_1]$ denotes the class of $L_1$ on $U_{\overline{k}}$. Now let
\begin{equation*}
    \hat{L_1}:= \{x_0^2+\alpha x_1^2=x_2^2+\beta \omega=0\}\subseteq \mathcal{S}_{\overline{k}}, 
\end{equation*}
then by applying $f^*$ to $\phi$ we obtain the $1$-cochain $f^*\phi:\Gal(k(\gamma)/k)\rightarrow \Pic V_{\overline{k}}$ defined by
\begin{equation*}
    \phi(1)=[0], \hspace{5mm} \phi(\sigma)=[\hat{L_1}\cap V], 
\end{equation*}
where $[\hat{L_1}\cap V]$ denotes the class of the divisor associated to $\hat{L_1}$ on $V_{\overline{k}}$. Now let $\psi:\Gal(k(\gamma)/k)\rightarrow \Pic\mathcal{S}_{\overline{k}}$ be the $1$-cocycle defined by
\begin{equation*}
    \psi(1)=[0], \hspace{5mm} \psi(\sigma)=[\hat{L_1}]-H,
\end{equation*}
where $H\in \Pic\mathcal{S}_{\overline{k}}$ is the hyperplane class. We claim that $[\hat{L_1}]$ is $\Gal(\overline{k}/k(\gamma))$-invariant and $\psi$ is a cochain. If these claims hold, then $[\psi]\neq0\in \operatorname{H}^1(\Gal(k(\gamma)/k),(\Pic\mathcal{S}_{\bar{k}})^{\Gal(\overline{k}/k(\gamma))})$. Thus since the inflation map is injective, we get a cochain $\inf(\psi) \neq0 \in \operatorname{H}^1(k,\Pic\mathcal{S}_{\overline{k}})$. Clearly $\res{\inf(\psi)}{V}=f^*\inf(\phi)$, and hence our result holds if these claims hold.\par
We firstly prove $[\hat{L_1}]$ is $\Gal(\overline{k}/k(\gamma))$-invariant. For any element $\theta\in \Gal(\overline{k}/k(\gamma))$, we have $\theta \cdot [\hat{L_1}]\in\{[\hat{L_1}],[\hat{L_1'}]\}$, where
\begin{equation*}
    \hat{L_1'}:=\{x_0^2-\alpha x_1^2=x_2^2-\beta \omega=0\}\subseteq S_{\overline{k}}.
\end{equation*}
Thus it suffices to prove $[\hat{L_1}]=[\hat{L_1'}]$. Let $\alpha':=\sqrt{-\frac{a_2}{a_0}}$, then we have
\begin{equation}
\label{eq: divisors for brauer group to show invariance 1}
    \left[\hat{L_1}\right]\bigcup\left[\{x_0^2-\alpha'x_2^2=\alpha x_1^2-\alpha'\beta \omega=0\}\right]\subseteq \left[\{x_0^2-\alpha x_1^2-\alpha' x_2^2 +\alpha' \beta \omega=0\}\right]=:X_1.
\end{equation}
Substituting $\omega= \frac{-1}{\alpha' \beta}\left(x_0^2-\alpha x_1^2-\alpha' x_2^2\right)$ into $\mathcal{S}_{\overline{k}}$ gives the equation
\begin{align*}
   0=2(x_0-\sqrt{\alpha}x_1)(x_0+\sqrt{\alpha}x_1)(x_0-\sqrt{\alpha'}x_2)(x_0+\sqrt{\alpha'}x_2).
\end{align*}
In particular $X_1$ has the same $4$ irreducible components as the left-side of \eqref{eq: divisors for brauer group to show invariance 1}, and hence they must be equal. Therefore
\begin{equation}
\label{eq: first divisor equality for brauer group to show invariance}
    X_1=\left[\hat{L_1}\right]+\left[\{x_0^2-\alpha'x_2^2=\alpha x_1^2-\alpha'\beta \omega=0\}\right].
\end{equation}
By the same argument as above we have
\begin{equation}
\label{eq: second divisor equality for brauer group to show invariance}
    X_2=\left[\hat{L_1'}\right]+\left[\{x_0^2-\alpha'x_2^2=\alpha x_1^2-\alpha'\beta \omega=0\}\right],
\end{equation}
where 
\begin{equation*}
    X_2:=\left[\{x_0^2+\alpha x_1^2-\alpha'x_2^2-\alpha'\beta\omega=0\}\right]\in \Pic \mathcal{S}_{\overline{k}}.
\end{equation*}
Combining equations~\eqref{eq: first divisor equality for brauer group to show invariance} and \eqref{eq: second divisor equality for brauer group to show invariance} gives the desired result. \par
We now prove $\psi$ is a cochain. Let
\begin{equation*}
    \hat{L_2}:= \{x_0^2-\alpha x_1^2=x_2^2+\beta\omega=0\}\subseteq \mathcal{S}_{\overline{k}}.
\end{equation*}
Clearly we have
\begin{equation*}
   \left[\{x_0^2-\alpha x_1^2=0\}\right]= \left[\hat{L_1}\right]+\left[\hat{L_2}\right].
\end{equation*}
Therefore since we have $\sigma [\hat{L_1}] = [\hat{L_2}]$, it follows that
\begin{equation*}
     \left[\hat{L_1}\right]+\sigma \left[\hat{L_1}\right] =\left[\{x_0^2-\alpha x_1^2=0\}\right]=2 H. \qedhere
\end{equation*}
\end{proof}

\bibliographystyle{amsalpha}{}
\bibliography{references}
\end{document}